\documentclass[10pt, a4paper]{amsart}
\usepackage[british]{babel}
\usepackage{amsmath}
\usepackage{amssymb}
\usepackage{amsthm}
\usepackage{graphicx}
\usepackage{url}
\usepackage{enumerate}
\usepackage{xcolor}
\usepackage{mathrsfs} 
\usepackage{t1enc}
\usepackage{mathtools}
\usepackage{placeins} 
\usepackage{tikz}
\usepackage{textcomp}
\usepackage{appendix}
\usetikzlibrary{matrix}
\usetikzlibrary{arrows,automata}
\usetikzlibrary{decorations.pathreplacing}
\usetikzlibrary{shapes.geometric}
\usepackage[colorlinks=true,allcolors=blue]{hyperref}

\newtheorem{thm}{Theorem}[section]
\newtheorem{lem}[thm]{Lemma}
\newtheorem{prop}[thm]{Proposition}
\newtheorem{obs}[thm]{Observation}
\newtheorem{cor}[thm]{Corollary}

\newtheorem{que}{Question}

\theoremstyle{definition}
\newtheorem{defn}[thm]{Definition}
\newtheorem{rem}[thm]{Remark}

\DeclareMathOperator{\supp}{supp} 

\DeclareMathOperator{\Int}{int}
\DeclareMathOperator{\dist}{dist}

\DeclareMathOperator{\diam}{diam}

\DeclareMathOperator{\scr}{cr}
\DeclareMathOperator{\Fl}{Fl}

\newcommand{\set}[1]{\left\{#1\right\}}

\newcommand{\eps}{\varepsilon}
\newcommand{\R}{\mathbb{R}}

\newcommand{\N}{\mathbb{N}}

\DeclareMathOperator{\id}{id}

\begin{document}

\title[Parametrized family of pseudo-arc attractors]{Parametrized family of pseudo-arc attractors: physical measures and prime end rotations}

\author{Jernej \v Cin\v c}
\address[J. \v Cin\v c]{Faculty of Mathematics, University of Vienna, Oskar-Morgenstern-Platz 1, A-1090 Vienna, Austria -- and -- National Supercomputing Centre IT4Innovations, University of Ostrava,
	IRAFM,
	30. dubna 22, 70103 Ostrava,
	Czech Republic}
\email{cincj9@univie.ac.at}

\author{Piotr Oprocha}
\address[P. Oprocha]{
	AGH University of Science and Technology, Faculty of Applied Mathematics,
	al. Mickiewicza 30,
	30-059 Krak\'ow,
	Poland -- and -- National Supercomputing Centre IT4Innovations, University of Ostrava,
	IRAFM,
	30. dubna 22, 70103 Ostrava,
	Czech Republic}
\email{oprocha@agh.edu.pl}
\begin{abstract}
	The main goal of this paper is to study topological and measure-theoretic properties of an intriguing family of strange planar attractors.
	Building towards these results, we first show that any generic Lebesgue measure-preserving map $f$ generates the pseudo-arc as inverse limit with $f$ as a single bonding map.
	These maps can be realized as attractors of disc homeomorphisms in such a way that the attractors vary continuously (in Hausdorff distance on the disc) with the change of bonding map as a parameter.
	Furthermore, for generic Lebesgue measure-preserving maps $f$ the background Oxtoby-Ulam measures induced by Lebesgue measure for $f$ on the interval are physical on the disc and in addition there is a dense set of maps $f$ defining a unique physical measure. Moreover, the family of physical measures on the attractors varies continuously in the weak* topology; i.e. the parametrized family is statistically stable.
	We also find an arc in the generic Lebesgue measure-preserving set of maps and construct a family of disk homeomorphisms parametrized by this arc which induces a continuously varying family of pseudo-arc attractors with prime ends rotation numbers varying continuously in $[0,1/2]$. It follows that there are uncountably many dynamically non-equivalent embeddings of the pseudo-arc in this family of attractors. 
\end{abstract}
\maketitle
\tableofcontents
\section{Introduction}
The main goal of this paper is to study topological and measure-theoretic properties of an intriguing family of strange planar attractors. Our study is  motivated by the advances of Wang and Young \cite{WY1,WY2} where they give an approach to study measure-theoretic properties of a large class of strange attractors with one direction of instability. Furthermore, there has been recent major advances by Boyland, de Carvalho and Hall who provided the first detailed description of a family of strange attractors arising from unimodal inverse limits both from topological \cite{3G-BM,BdCH1,BdCHInvent} and measure-theoretic perspective \cite{BdCH,BdCHNew}. The last mentioned results in particular focused on the family of tent inverse limits for which Barge, Bruin and \v Stimac \cite{BBS} have proven that the spaces are non-homeomorphic for different slopes in $(\sqrt{2},2]$. Our proposed (and studied) family of strange attractors also exhibits one direction of instability and several good measure-theoretic properties but is different from the last described family of tent inverse limits as follows. We provide a parametrized family of strange attractors where all the attracting sets are homeomorphic but nevertheless, as we shall see later, they exhibit a variety of rich dynamical behavior and have good measure-theoretic and statistical properties. The attracting sets of this family are all homeomorphic to the one-dimensional space of much interest in Continuum Theory and beyond, called the pseudo-arc. A {\em continuum} is a nonempty compact connected metric space. The pseudo-arc may be regarded as the most intriguing planar continuum not separating the plane. On the one hand its structure is quite complicated, since it does not contain any arc. On the other hand it reflects much regularity in its shape, since it is homeomorphic to any of its proper subcontinua.
For the history of the pseudo-arc and numerous results connecting it to other mathematical fields we refer the reader to the introduction in \cite{BCO}.  Our results here can also be viewed as a connecting link between Continuum Theory and Measure Theory since (among other results) we show that the natural extension (in the dynamically precise sense, cf. Subsection~\ref{subsec:results}) of topologically generic dynamics on the interval maps that preserve Lebesgue measure $\lambda$ lives on the pseudo-arc.

\subsection{Statements of the results}\label{subsec:results}
In what follows, by a {\em residual} set we mean a dense $G_{\delta}$ set and we call a property {\em generic} if it is satisfied on at least a residual set of the underlying Baire space.
In this subsection we will state and comment the main results of this paper. In the first part of the paper we focus our study on the class of continuous interval maps that preserve the Lebesgue measure $\lambda$ which we denote by $C_{\lambda}(I)$. If one equips this space with the metric of uniform convergence it becomes a complete space (see e.g. Proposition~4 in \cite{BCOT}). The study of properties of generic maps of $C_{\lambda}(I)$ was initiated in \cite{B} and continued recently in \cite{BT} and \cite{BCOT}; among other results it was proven in \cite{BT} that the generic maps are locally eventually onto (leo) and measure-theoretically weakly mixing. In \cite{BCOT} the authors focused on periodic properties of the generic maps and, among other results, they completely characterized their sets of periodic points of any period, determined their Hausdorff and upper box dimension, and proved that these maps have the shadowing and periodic shadowing property.\\
Here we prove another topological property of Lebesgue measure-preserving maps, which might be the most surprising of the properties yet; namely we prove:

\begin{thm}\label{thm:UniLimPresLeb}
There is  a dense $G_\delta$ set $\mathcal{T}\subset C_{\lambda}(I)$ such that if $f\in \mathcal{T}$ then
for every $\delta>0$ there exists a positive integer $n$ so that $f^n$ is $\delta$-crooked. 
\end{thm}

The $\delta$-crookedness is not an easy-to-state property (see Definition~\ref{def:crooked}), since it imposes strong requirements on values of the map. However, $\delta$-crookedness in the sense of Theorem~\ref{thm:UniLimPresLeb} completely characterizes the maps for which the inverse limit is the pseudo-arc (\cite{BingPacific} and Proposition 4 in \cite{MT}). Thus we obtain the following corollary.

\begin{cor}\label{cor:MincTransue}
The inverse limit with any $C_{\lambda}(I)$-generic map as a single bonding map is the pseudo-arc.
\end{cor} 

One should note that all interesting ''global`` dynamics in interval maps can be reflected in Lebesgue measure-preserving maps (see e.g. Remark in \cite{BCOT} or Remark~\ref{rem:equivalent} below); for any non-atomic measure on the interval with full support, we obtain the result analogous to Theorem~\ref{thm:UniLimPresLeb}.

At this point, let us mention that one can view inverse limits with single bonding maps as the simplest invertible dynamical extensions of the dynamics given by the bonding map. Let us state this fact more precisely. Denote by $\hat I:=\underleftarrow{\lim}(I,f)$ and let $\hat f\colon \hat I\to \hat I$ be the {\em natural extension of $f$} (or the {\em shift homeomorphism}). A {\em natural projection} $\pi_0:\hat I\to I$ defined by $\pi_0(x)=x_0$ {\em semi-conjugates} $\hat{f}$ to $f$.
	\begin{figure}[!ht]
		\centering
		\vspace{-0.5cm}
		\begin{tikzpicture}[->,>=stealth',auto, scale=1.3]
		\node (1) at (0,0) {$I$};
		\node (2) at (1,0) {$I$};
		\node (3) at (1,1) {$\hat{I}$};
		\node (4) at (0,1) {$\hat{I}$};
		\draw [->] (1) -- node[above]{\hspace{0.1cm}\small $f$} (2);
		\draw [->] (3) -- node[right]{\small $\pi_0$} (2);
		\draw [->] (4) -- node[above]{\small $\hat f$} (3);
		\draw [->] (4) -- node[left]{\small $\pi_0$} (1);
		\end{tikzpicture}
	\end{figure}
	
	For a continuum $Y$ let $g:Y\to Y$ be an invertible dynamical system and let $p:Y\to I$ factor $g$ to $f$. Then $p$ factors through $\pi_0$: i.e. $\hat f$ is the minimal invertible system which extends $f$. 
	
	\begin{figure}[!ht]
		\centering
		\vspace{-0.45cm}
		\hspace{1.2cm}
		\begin{tikzpicture}[->,>=stealth',auto, scale=1.3]
		\node (1) at (0,0) {$I$};
		\node (2) at (1,0) {$I$};
		\node (3) at (1,1) {$\hat{I}$};
		\node (4) at (0,1) {$\hat{I}$};
		\node (5) at (1,2) {$Y$};
		\node (6) at (0,2) {$Y$};
		\draw [->] (1) -- node[above]{\hspace{0.1cm}\small $f$} (2);
		\draw [->] (3) -- node[right]{\small $\pi_0$} (2);
		\draw [->] (4) -- node[above]{\small $\hat f$} (3);
		\draw [->] (4) -- node[left]{\small $\pi_0$} (1);
		\draw [->] (5) -- node[right]{\small $\pi$} (3);
		\draw [->] (6) -- node[left]{\small $\pi$} (4);
		\draw [->] (6) -- node[above]{\small $g$} (5);
		\draw [->,red] (5.north) to [out=150,in=30] (2.north);
		\draw [->,red] (6.north) to [out=150,in=30] (1.north);
		\node at (2,1) {\small {\color{red} $p=\pi_0\circ \pi$}};
		\end{tikzpicture}
		\caption{$\hat f$ is the minimal invertible system which extends $f$.}\label{diagram}
	\end{figure}

 Denote by $C_{DP}(I)\subset C(I)$ the class of interval maps with the dense set of periodic points and by $\overline{C_{DP}(I)}$ its closure. Building on the well known properties of interval maps mentioned earlier (see Remark~\ref{rem:equivalent}) we obtain the following result.

\begin{cor}\label{cor:MincTransue1}
The inverse limit with any $\overline{C_{\textrm{DP}}(I)}$-generic map as a single bonding map is the pseudo-arc.
\end{cor}

Interval maps with dense set of periodic points were popularized by the work of Li and Yorke \cite{LiYorke} where such maps were called ``chaotic'' for the first time. This line of work saw numerous applications in different branches of mathematics and beyond. Our last result above can also be viewed as a continuation of study initiated by Barge and Martin \cite{BM1,BM2,BM3} in the generic setting. 
Corollaries~\ref{cor:MincTransue} and \ref{cor:MincTransue1} seem quite unexpected, taking into account earlier genericity results about inverse limits of interval maps; in particular, it was proven by Block, Keesling and Uspenskij \cite{BKU} that the set of interval maps that produce pseudo-arc in the inverse limit are nowhere dense in $C(I)$ (where $C(I)$ denotes the class of all continuous interval maps). On the other hand, Bing \cite{BingPacific} has shown that for any manifold $M$ of dimension at least $2$, the set of subcontinua homeomorphic to the pseudo-arc is a dense residual subset of the set of all subcontinua of $M$ (equipped with the Vietoris topology).

Inverse limit spaces are often not Euclidean spaces and thus it usually (also often in our case) makes no sense to speak about Lebesgue measure on the inverse limit. However, any invariant probability measure lifts to a shift-invariant measure on the inverse limit space (see \cite{KRS}). In particular, if we have a locally eventually onto bonding map on a Euclidean space, then the measure on the inverse limit can be seen as an extension of the measure on the underlying Euclidean space over Cantor set fibers. Precise definitions of these concepts are given later in the paper (see Definition~\ref{def:induced}). In standard terms,  a measure $\mu$ on a manifold is {\em physical} if the set of its regular points of $\mu$ has positive measure with respect to a background Lebesgue measure. 
It was proven in \cite{KRS} that if an Euclidean space admits a physical measure, the shift-invariant measure on the inverse limit space is also physical. 
If we combine the last theorem, corollary and the results from \cite{BT}, \cite{Li} and \cite{Br} (see also the survey \cite{CL} on dynamical properties that extend to inverse limit spaces) we get the following. Note that the following corollary also contributes to the study of possible homeomorphisms on the pseudo-arc.

\begin{cor}\label{cor:Typical}
Let $\mathcal{T}$ be a dense $G_{\delta}$ subset of $C_{\lambda}(I)$ from Theorem~\ref{thm:UniLimPresLeb}.
There is a dense $G_{\delta}$ subset $\mathcal{T}'\subset \mathcal{T}\subset C_{\lambda}(I)$ so that for every $f\in \mathcal{T}'$ the inverse limit $\hat I_f:=\underleftarrow{\lim}(I, f)$ is the pseudo-arc and the natural extensions of maps from $\mathcal{T}'$ give rise to complete space $\hat{\mathcal{T}'}$ of homeomorphisms on the pseudo-arc $\hat I_f$ so that every $\hat f\in \hat{\mathcal{T}}'$:
\begin{enumerate}
 \item preserves induced inverse limit $\hat{m}$ measure on $\hat I_f$,
 \item induced inverse limit measure $\hat{m}$ is physical and weakly mixing on $\hat I_f$,
 \item\label{cor:item4}  is transitive,
 \item has infinite topological entropy,
 \item  has the shadowing property,
 \item has a Cantor set of periodic points of any period. 
\end{enumerate}
\end{cor}

Note that the preceding result works also for generic maps in the class of maps preserving any other fully supported probability measure absolutely continuous with respect to the Lebesgue measure.

The results above serve as the preparatory results for our study of a family of strange attractors. The tool that we apply is the so-called Brown-Barge-Martin (BBM) embedding of inverse limits of topological graphs (see~\cite{BM} and~\cite{Br}). This approach yielded surprising new examples for topological dynamical systems as we explain in this paragraph. 
A particularly useful extension of this method is provided by the parametrized version of BBM embedding (and we will use this method in the following theorem but not in the main theorem of the paper), given by Boyland, de Carvalho and Hall \cite{3G-BM}. The same authors used this method as a tool to find new rotation sets for torus homeomorphisms (see \cite{BdCHInvent}) and to study prime ends of natural extensions of unimodal maps (see \cite{BdCH}). Very recently, Boro\'nski, \v Cin\v c and Liu used an adaptation of the BBM technique to provide several new examples in the dynamics on the $2$-sphere, with the particular emphasis on better understanding the induced boundary dynamics of invariant domains in parametrized families. 
The above mentioned BBM technique enables us to present the inverse limit of any interval map as a planar attractor. The problem with the standard approach, however, is that BBM embeddings done for two maps separately may be incomparable. In fact, it may happen in practice that arbitrarily close maps may generate
	quite distant attractors (e.g. in terms of Hausdorff distance) and also the other extreme is possible. While we use a large collection of different interval maps, their
	inverse limit is the pseudo-arc, that is, these inverse limits are homeomorphic. It may therefore happen, that BBM results in conjugate systems,
	i.e. two maps define dynamically the same systems.
	The following result allows comparison of attractors, providing 
	continuous dependence between the shape of attractors and distance between interval maps inducing them. It also ensures that we construct
	numerous non-conjugate dynamical systems.
	
Denote by $C(X,Y)$ (respectively, $\mathcal{H}(X,X)$) the set of all continuous mappings from a metric space $X$ to a metric space $Y$ (respectively, the set of all homeomorphisms of $X$). We equip the space $C(X,Y)$ with the metric of uniform convergence $\rho$. Let $D\subset \mathbb{R}^2$ denote a closed topological disk. We say that a compact set $K\subset D\subset \mathbb{R}^2$ is the {\em (global) attractor} of $h\colon D\to D$ in $D$  if for every $x\in D \setminus \partial D$, the {\em omega limit set} {$\omega_h(x)\subset K$  and for some $z\in K$ we have that $\omega_h(z)=K$.} 

To a non-degenerate and non-separating continuum $K\subset D\setminus \partial D\subset \mathbb{R}^2$ we can associate the {\em circle of prime ends} as the compactification of $D\setminus K$.
If $\tilde{h}:\mathbb{R}^2\to \mathbb{R}^2$ preserves orientation and $\tilde h(K)=K$, $\tilde h(D)=D$ then $\tilde h$ induces an orientation preserving homeomorphism of the prime ends circle, and therefore it gives a natural \emph{prime ends rotation number}. 
The prime ends rotation number allows one to study boundary dynamics of underlying global attractors and distinguish their embeddings from dynamical point of view, see Definition~\ref{def:equivalent} and the remark thereafter. For a more comprehensive introduction to the prime end theory we refer the reader to \cite{Mather}. Let us note that we will not delve deep in this line of research in the current paper, but nevertheless, the extensions of the following results in this direction would, in our opinion, be of interest. Let us also note that the following theorem was in part motivated by the results obtained in \cite{BCL}, although the main reason that we provide it here is to give an insight in the topological properties of some embeddings that we later study in Theorem~\ref{lem:BBM} from the measure-theoretic aspect.

\begin{thm}\label{lem:BBM1}
Let $\mathcal{T}$ be a dense $G_{\delta}$ subset of $C_{\lambda}(I)$ from Theorem~\ref{thm:UniLimPresLeb}.
 There is a parametrized family of interval maps $\{f_t\}_{t\in [0,1]}\subset \mathcal{T}\subset C_{\lambda}(I)$ and a parametrized family of homeomorphisms $\{\Phi_{t}\}_{t\in[0,1]}\subset \mathcal{H}(D, D)$ varying continuously with $t$ having $\Phi_t$-invariant pseudo-arc attractors $\Lambda_t\subset D$ for every $t\in [0,1]$ so that 
	\begin{itemize}
		\item[(a)]\label{item(a)}  $\Phi_t|_{\Lambda_t}$ is topologically conjugate to $\hat{f}_{t}\colon \hat I_f\to \hat I_f$.
		\item[(b)]\label{item(b)} The attractors $\{\Lambda_t\}_{t\in [0,1]}$ vary continuously in the Hausdorff metric. 
		\item[(c)]\label{item(c)} Prime ends rotation numbers of homeomorphisms $\Phi_{t}$ vary continuously with $t$ in the interval $[0,1/2]$.
		\item[(d)] There are uncountably many dynamically non-equivalent planar embeddings of the pseudo-arc in the family $\{\Lambda_t\}_{t\in[0,1]}$.
		\end{itemize}
\end{thm}

This result is interesting also from several other aspects. First, it is the first example in the literature (to our knowledge) of a parametrized family of strange attractors where the attractors are proven to be homeomorphic, yet the boundary dynamics on the attractor is very rich. This result underlines the fact that pseudo-arc is among one-dimensional continua a special object with respect to its flexibility to permit a variety of different dynamical behavior. Furthermore, let us also note that Theorem~\ref{lem:BBM1} answers Question~2 from \cite{BCL} for the case of the pseudo-arc, however these results do not directly apply to the pseudo-circle.
Second, the above result also says there are pathwise-connected components in a generic set of $C_{\lambda}(I)$ and it would be interesting to know if this set itself is pathwise-connected.
Third, let us mention that the homeomorphism group of the pseudo-arc contains no non-degenerate continua \cite{LewisGroups}, so Theorem~\ref{lem:BBM1}
may come as a surprise, since it defines a continuous family of homeomorphisms $\Phi_t|_{\Lambda_t}$ where each $\Lambda_t$ is the pseudo-arc.
However, there is no contradiction with results of \cite{LewisGroups} in statements of Theorem~\ref{lem:BBM1}, since each $\Lambda_t$ is the pseudo-arc $P$ up to a homeomorphism $h_t\colon \Lambda_t\to P$.
But first of all, there is no reason why the family $\{h_t\}_{t\in [0,1]}$ should be continuous.
There is also no an immediate argument why homeomorphisms
$h_t\circ \Phi_t|_{\Lambda_t}\circ h_t^{-1}$ are different. 
Since in this work we are interested only in embeddings from dynamical perspective we leave the following question about the topological nature of embeddings open.

\begin{que}\label{q:1}
Are for every $t\neq t'\in [0,1]$ the attractors $\Lambda_t$ and $\Lambda_{t'}$ (topologically) non-equivalently embedded? 
\end{que}
Recall that two planar embeddings are called {\em (topologically) equivalent} if there is a homeomorphism $h\colon \mathbb{R}^2\to \mathbb{R}^2$ such that $h(\Lambda_t)=\Lambda_{t'}$.
Note that $h$ does not have to intertwine the dynamics, so Theorem~\ref{lem:BBM1} does not provide the answer to the above question.
We leave it as a problem for future research.

The last part of this paper is the study of measure-theoretic properties of BBM embeddings of attractors obtained as inverse limits of generic maps in $C_{\lambda}(I)$, and we take Theorem~\ref{thm:UniLimPresLeb} as a starting point. Therefore, besides the topological input given by Theorem~\ref{thm:UniLimPresLeb} it turns out that this family is particularly nice also from the measurable and statistical perspective. A Borel probability measure on a manifold $M$ is called {\em Oxtoby-Ulam (OU)} or {\em good} if it is non-atomic, positive on open sets, and assigns zero measure to the boundary of manifold $M$ (if it exists) \cite{APBook,OxtobyUlam}. As we mentioned earlier, using the BBM technique \cite{BM}, we can represent any inverse limit of interval map as an attractor of a disc homeomorphism. Repeating a simplified version of approach in \cite{3G-BM} we can easily ensure that any invariant measure becomes a physical measure. However, if  this construction is performed for each map separately we would not be able to ensure comparability of obtained embeddings. 
Such an approach would not be satisfactory since it is natural to require from the embedding technique that ``similar'' maps result in ``similar'' embeddings. 
An important result of this type also ensuring statistical stability of attractors was first obtained for the tent inverse limit family in \cite{3G-BM} (see also \cite{BdCH}) and was an inspiration for the following theorem. However, note also that such a result is by no means given beforehand; taking e.g. logistic family instead of the tent map family one cannot prove statistical stability of homeomorphisms obtained from the BBM construction (see \cite{3G-BM} for more detail).

\begin{thm}\label{lem:BBM}
Let $\mathcal{T}$ be a dense $G_{\delta}$ subset of $C_{\lambda}(I)$ from Theorem~\ref{thm:UniLimPresLeb}.
 There exists a dense $G_{\delta}$ set of maps $A\subset\mathcal{T}\subset C_{\lambda}(I)$ and a parametrized family of homeomorphisms $\{\Phi_f\}_{f\in A}\subset \mathcal{H}(D, D)$ varying continuously with $t$ having $\Phi_f$-invariant pseudo-arc attractors $\Lambda_f\subset D$ for every $f\in A$ so that 
 	\begin{itemize}
		\item[(a)]\label{item:a}  $\Phi_f|_{\Lambda_f}$ is topologically conjugate to $\hat{f}\colon \hat I_f\to \hat I_f$.
		\item[(b)]\label{item:b} The attractors $\{\Lambda_f\}_{f\in A}$ vary continuously in Hausdorff metric. 
		\item[(c)]\label{item:c} The attractor $\Lambda_f$ supports induced weakly mixing measure $\mu_{f}$ invariant for $\Phi_f$ for any $f\in A$. Let $\lambda_f$ be an induced Oxtoby-Ulam (OU) measure on $D$. There exists an open set $U\subset D$ which for each $f$ contains $U_{f}\subset U$ so that $\lambda_f(U_f)=\lambda(U)$ and $U_f$ is in the basin of attraction of $\mu_f$. In particular each $\mu_f$ is a physical measure. 
		\item[(d)]\label{item:d} There exist a dense countable set of maps $\{g_i\}^{\infty}_{i=0}\subset A$ for which $\mu_{g_i}$ is the unique physical measure, i.e. its basin of attraction has the full $\lambda_{g_i}$-measure in $D$.
		\item[(e)]\label{item:e} $\Phi_f|_{\Lambda_f}$ is transitive and has the shadowing property.
		\item[(f)] Measures $\mu_f$ vary continuously in the weak* topology, i.e. family $\{\Phi_f\}_{f\in A}$ is statistically stable.
	\end{itemize}
\end{thm}
One of the difficulties to obtain the above result is that the space $C_\lambda(I)$ is not compact and so we cannot apply results for BBM parametrized families approach from \cite{3G-BM,BdCH} directly; therefore we have to provide our own construction. 
 We provide a new version of parametrized BBM construction that works for particular complete parameter spaces which helps us to obtain properties (a) and (b) from Theorem~\ref{lem:BBM} and at the same time ensures that $\{\Phi_f\}_{f\in A}\subset \mathcal{H}(D, D)$ vary continuously with $t$. Furthermore, our construction also controls measures of the sets that are attracted to $\Lambda_f$. This is obtained by adjusting Oxtoby-Ulam technique of approximating the space by Cantor sets and move them around in a controlled fashion. This allows us to obtain properties (c) and (d) where, in particular, property (d) requires a very careful control, and as such, one cannot expect it can be extended onto all the maps in the family using our approach. We also show that our family of attractors behaves well from the statistical point of view; namely, the induced measures on these attractors vary continuously in the weak* topology, or in other words our family is statistically stable. 
 
It would be of great interest to obtain similar result as in Theorem~\ref{lem:BBM} using the $\mathcal{C}^1$ topology on the disk. Such a task seems to be very hard and we cannot hope to obtain the result adjusting the BBM approach. 
{The generic interval maps that we deal with are nowhere differentiable. Furthermore, we cannot use $\mathcal{C}^1$ topology instead,}  since that would imply finite entropy (cf. \cite{MP}), while we know that positive entropy interval maps giving the pseudo-arc as inverse limit must have infinite entropy \cite{Mouron}. On the other hand, recent advances show that finite entropy is possible on the pseudo-arc \cite{BCO} and pseudo-arc is the typical continuum in $\R^2$ (see \cite{BingPacific}) which gives a chance for a positive solution. 

 An important motivation for studying statistical stability of attractors originated from much earlier works than \cite{BdCH1}. The concept of statistical persistence of some phenomena was originally defined by Alves and Viana \cite{AV} and it expresses the continuous variation of physical measures as a function of the evolution law governing the systems. A natural testing ground for this concept was a well known parametrized family of H\'enon attractors.
 This line of research culminated in the work of Alves, Carvalho and Freitas \cite{ACF} who proved that H\'enon maps for Benedicks-Carleson parameters \cite{BC} are indeed statistically stable. However, H\'enon attractors are in some sense very fragile. This is supported by the result of Ures \cite{Ures} who showed that the Benedicks–Carleson parameters can be approximated by other parameters for which the H\'enon map has a homoclinic tangency associated to a fixed point. Hence, using the Newhouse’s results \cite{New1,New2}, one can deduce the appearance of infinitely many attractors in the neighborhood of the H\'enon attractors for Benedicks-Carleson parameters. 
Our Theorem~\ref{lem:BBM} ensures statistical stability for 
the considered family, however only for (topologically) small set of maps we obtain unique physical measure. Its uniqueness is hard to reproduce for the whole family, and we cannot exclude the situation, that similarly to H\'enon attractors, several physical measures will appear when arbitrarily small perturbation is applied.

\subsection{Insight into the proof and the outline of the paper}

For preliminary results concerning crookedness we adjust in Section~\ref{sec:crooked} techniques developed by Minc and Transue \cite{MT} and combine them with a special window perturbations that were first introduced in \cite{BT} and subsequently used in \cite{BCOT,BCOT1}. Of central importance in proving Theorem~\ref{thm:UniLimPresLeb} is Lemma~\ref{lem:MincUpdt}, where we show that the Lebesgue measure-preserving perturbations we construct satisfy certain requirements from \cite{MT}. This allows us to apply important techniques developed therein. We provide this construction in Section~\ref{sec:crooked}, prove Theorem~\ref{thm:UniLimPresLeb} and then extend the argument on the closure of the class of interval maps with dense set of periodic points. We use Section~\ref{sec:extension} to harvest the low hanging fruit through the inverse limit construction (in particular we obtain Corollaries~\ref{cor:MincTransue} and \ref{cor:Typical}).

The second major part of this paper is the proof of Theorem~\ref{lem:BBM1} in Section~\ref{sec:arc}. We start with a continuously varying family of piecewise affine Lebesgue measure-preserving interval maps with slope large enough and obtain maps that satisfy Theorem~\ref{thm:UniLimPresLeb}. The proof of Theorem~\ref{lem:BBM1} can be compared with the proof of Theorem~\ref{thm:UniLimPresLeb}, with the main difference being that we provide a sequence of special perturbations that are appropriate for the whole family of interval maps. We need to note at this point that besides the requirement on the lower bound for slopes of these maps there is nothing special about our chosen family of interval maps; we could easily repeat the procedure starting with a different (non-conjugate) piecewise affine family of Lebesgue measure-preserving interval maps. However, we have no tools to prove that the new family obtained after perturbations would be different (i.e. maps would not be topologically conjugate) to the original family we have chosen to study. To make attractors out of this continuously varying family of maps we can directly apply machinery developed in \cite{3G-BM} and combining this with the result of Barge \cite{Barge} we get the required result on the continuity of the prime ends rotation numbers of the attractors.

The last major part of the paper is Section~\ref{sec:BBM2} where we prove Theorem~\ref{lem:BBM}. For that purpose we develop modifications of the BBM embeddings technique which are required since our parameter space is complete but not compact. Therefore, we combine continuously varying BBM technique (unwrappings in the language of \cite{3G-BM}) with direct application of tools from the proof of Brown's theorem \cite{Br} which we extend for our particular class of generic Lebesgue measure-preserving interval maps. To make these embedding measure-theoretically interesting we combine the technique with an adaptation of Oxtoby-Ulam technique of controlled transformations of a dense collection Cantor sets in a topological disk; the rest of the proof shows how to obtain all the items from Theorem~\ref{lem:BBM} which is indeed possible due to the inverse limit construction implemented in the BBM technique. An interesting question that we did not address for the family of attractors constructed in this section is in how many equivalence classes planar embeddings of attractors fall into\footnote{Theorem~\ref{lem:BBM1} sheds a light with respect to this problem, however since $A$ is a proper subset of $\mathcal{T}$ we, a priori, cannot claim anything in this direction.}. While the attractors themselves are topologically the same (all pseudo-arcs are homeomorphic), the homeomorphism between them does not necessarily extend to the disc. The answer strongly depends on construction in the proof of Theorem~\ref{lem:BBM}, especially the properties of map $\Theta_f$ identifying constructed inverse limit with the disc.

\section{Crookedness is generic in the family of Lebesgue measure-preserving interval maps}\label{sec:crooked}
\subsection{General preliminaries}
Let $\mathbb{N}:=\{1,2,3,\ldots\}$ and $\mathbb{N}_0:=\mathbb{N}\cup\{0\}$. Let $I:=[0,1]\subset \mathbb{R}$ denote the unit interval. Let $\mathrm{diam}(A)$ denote the diameter of $A\subset I$.
Let $\lambda$ denote the \emph{Lebesgue measure} on the underlying Euclidean space. By $C(I)$ we denote the family of all continuous interval maps. Furthermore, let $C_{\lambda}(I)\subset C(I)$ denote the family of all continuous Lebesgue measure-preserving functions of $I$. We equip both $C(I)$ and $C_{\lambda}(I)$ with the \emph{metric of uniform convergence} $\rho$:
$$\rho (f,g) := \sup_{x \in I} |f(x) - g(x)|.$$
For a metric space $(X,d)$ we shall use $B(x,\xi)$ for the open ball of radius $\xi$ centered at $x\in X$ and for a set $U\subset X$ we shall denote 
$$B(U,\xi):=\bigcup_{x\in U}B(x,\xi).$$ In the rest of the paper we use the letter $d$ to denote the \emph{Euclidean distance} on the underlying Euclidean space.
We say that a map $f$ is \emph{piecewise linear (or piecewise affine)} if it has finitely many \emph{critical points} (i.e. points $x\in I$ such that $f|_{J}$ is not one-to-one for every open interval $x\in J\subset I$) and is linear on every {\em interval of monotonicity} (an interval $J\subset I$ on which $f$ is monotone, but is not monotone on any interval properly containing $J$). We say that an interval map $f$ is \emph{locally eventually onto (leo)} if for every open interval $J\subset I$ there exists a non-negative integer $n$ so that $f^n(J)=I$. This property is also sometimes referred in the literature as \emph{topological exactness}.

\subsection{Proof of Theorem~\ref{thm:UniLimPresLeb}}
For an easier visualization of the concept defined in the following definition we refer the reader to Figure~\ref{fig:sigma}, where examples of such maps are given.

\begin{defn}\label{def:crooked}
	Let $f\in C(I)$, let $a,b\in I$ and let $\delta>0$. We say that $f$ is {\em $\delta$-crooked between $a$ and $b$} if for every two points $c, d \in I$ such that $f(c) = a$ and $f(d) = b$, there is a point $c'$ between $c$ and $d$ and there is a point $d'$ between $c'$ and $d$ such that $|b - f(c')| < \delta$ and $|a - f(d')| < \delta$. We will say that $f$ is {\em $\delta$-crooked} if it is $\delta$-crooked between every pair of points.
\end{defn}

The following two definitions and the first part of the third definition were introduced in \cite{LM}.
We will use the maps defined below as the building blocks of our perturbations.

\begin{defn}
Let $(\scr [n])^\infty_{n=1}\subset \N$ be the sequence defined in the following way: 
$\scr [1]:=1$, $\scr [2]:=2$ and $\scr [n]:=2\scr [n-1]+\scr [n-2]$ for each $n\geq 3$.
\end{defn}

\begin{defn}
Let $g_1$ and $g_2$ be two maps of $I$ into itself such that $g_1 (0) = g_2(0)=0$ and $g_1(1)=g_2(1)=1$. Suppose $m\geq 3$ is an integer and $s$ is a real number such that $0 < s < 1/2$. Then $\phi [g_1, g_2, s, m]$ is the function of $I$ into itself defined by the formula:
\begin{equation}
\label{funcphi}
\phi [g_1, g_2, s, m](t):=\begin{cases}
\frac{m-1}{m} g_2(\frac{t}{s}),& \text{if } 0\leq t\leq s,\\
\frac{1}{m}+\frac{m-2}{m}g_1(\frac{1-s-t}{1-2s}),& \text{if }s \leq t \leq 1 - s,\\
\frac{1}{m}+\frac{m-1}{m}g_2(\frac{t+s-1}{s}),& \text{if }1 -s\leq t\leq 1.
\end{cases}
\end{equation}
\end{defn}

\begin{defn}
For each integer $n\geq 3$ denote
$$
s_n :=\scr [n-1]/(2\scr [n-1]+\scr [n-2]).
$$
For each $n\in \N$, let \textit{simple $n$-crooked map}
$\sigma_n\colon I\to I$ be defined in the following way (see also Figure~\ref{fig:sigma} and Figure~\ref{fig:hatlambda}): 
\begin{itemize}
\item $\sigma_1 =\sigma_2$ is the identity on $I$, and
\item $\sigma_n :=\phi[\sigma_{n-2},\sigma_{n-1},s_n,n]$ for each positive integer $n\geq 3$.
\end{itemize}
Let $\sigma_{-n}$ denote the reflection of the simple $n$-crooked map, that is 
$\sigma_{-n}(t):= 1 -\sigma_n (t)$ for each $t\in I$. Let $\sigma_{n}^{L}:=\sigma_{n}|_{[0,1/2]}$ where $\sigma_{n}|_{[0,1/2]}:[0,1/2]\to [0,\frac{n-1}{n}]$,  $\sigma_{n}^{R}:=\sigma_{n}|_{[1/2,1]}$ where $\sigma_{n}|_{[1/2,1]}:[1/2,1]\to [\frac{1}{n},1]$.  Similarly as above let $\sigma_{-n}^L$ and $\sigma_{-n}^R$ denote the reflections of $\sigma_{n}^L$ and $\sigma_{n}^R$ respectively. 
\end{defn}

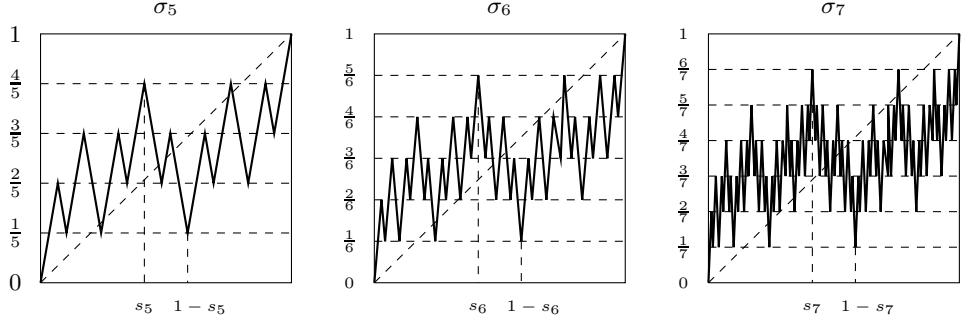
\begin{figure}[!ht]
	\centering
	\begin{tikzpicture}[scale=3.3]
	\draw (0,0)--(0,1)--(1,1)--(1,0)--(0,0);
	\draw[dashed] (0,0)--(1,1);
	\draw[dashed] (0,1/5)--(1,1/5);
	\draw[dashed] (0,2/5)--(1,2/5);
	\draw[dashed] (0,3/5)--(1,3/5);
	\draw[dashed] (0,4/5)--(1,4/5);
	\node at (-0.1,0) {\small $0$};
	\node at (-0.1,1/5) {\small $\frac{1}{5}$};
	\node at (-0.1,2/5) {\small $\frac{2}{5}$};
	\node at (-0.1,3/5) {\small $\frac{3}{5}$};
	\node at (-0.1,4/5) {\small $\frac{4}{5}$};
	\node at (-0.1,1) {\small $1$};
	\draw[thick] (0,0)--(2/29,2/5)--(3/29,1/5)--(5/29,3/5)--(7/29,1/5)--(9/29,3/5)--(10/29,2/5)--(12/29,4/5)--(14/29,2/5)--(15/29,3/5)--(17/29,1/5)--(19/29,3/5)--(20/29,2/5)--(22/29,4/5)--(24/29,2/5)--(26/29,4/5)--(27/29,3/5)--(1,1);
	\draw[dashed] (12/29,4/5)--(12/29,0);
	\draw[dashed] (17/29,1/5)--(17/29,0);
	\node at (12/29,-0.1) {\tiny $s_5$};
	\node at (17/29+0.05,-0.1) {\tiny $1-s_5$};
	\node at (0.5,1.1) {\small $\sigma_5$};
	\end{tikzpicture}
		\hspace{0,3cm}
	\begin{tikzpicture}[scale=3.3]
		\draw[dashed] (0,0)--(1,1);
	\node at (0.5,1.1) {\small $\sigma_6$};
		\node at (-0.1,0) {\tiny $0$};
	\node at (-0.1,1/6) {\tiny $\frac{1}{6}$};
	\node at (-0.1,2/6) {\tiny $\frac{2}{6}$};
	\node at (-0.1,3/6) {\tiny $\frac{3}{6}$};
	\node at (-0.1,4/6) {\tiny $\frac{4}{6}$};
	\node at (-0.1,5/6) {\tiny $\frac{5}{6}$};
	\node at (-0.1,1) {\tiny $1$};
	\node at (29/70,-0.1) {\tiny $s_6$};
	\node at (41/70+0.05,-0.1) {\tiny $1-s_6$};
	\draw (0,0)--(0,1)--(1,1)--(1,0)--(0,0);
	\draw[dashed] (0,1/6)--(1,1/6);
	\draw[dashed] (0,2/6)--(1,2/6);
	\draw[dashed] (0,3/6)--(1,3/6);
	\draw[dashed] (0,4/6)--(1,4/6);
	\draw[dashed] (0,5/6)--(1,5/6);
	\draw[thick] (0,0)--(2/70,2/6)--(3/70,1/6)--(5/70,3/6)--(7/70,1/6)--(9/70,3/6)--(10/70,2/6)--(12/70,4/6)--(14/70,2/6)--(15/70,3/6)--(17/70,1/6)--(19/70,3/6)--(20/70,2/6)--(22/70,4/6)--(24/70,2/6)--(26/70,4/6)--(27/70,3/6)--(29/70,5/6)--(31/70,3/6)--(32/70,4/6)--(34/70,2/6)--(36/70,4/6)--(38/70,2/6)--(39/70,3/6)--(41/70,1/6)--(43/70,3/6)--(44/70,2/6)--(46/70,4/6)--(48/70,2/6)--(50/70,4/6)--(52/70,3/6)--(53/70,5/6)--(55/70,3/6)--(56/70,4/6)--(58/70,2/6)--(60/70,4/6)--(61/70,3/6)--(63/70,5/6)--(65/70,3/6)--(67/70,5/6)--(68/70,4/6)--(1,1);
	\draw[dashed] (29/70,5/6)--(29/70,0);
	\draw[dashed] (41/70,1/6)--(41/70,0);
	\end{tikzpicture}
	\hspace{0,3cm}
	\begin{tikzpicture}[scale=3.3]
	\draw[dashed] (0,0)--(1,1);
	\node at (0.5,1.1) {\small $\sigma_7$};
	\node at (-0.1,0) {\tiny $0$};
	\node at (-0.1,1/7) {\tiny $\frac{1}{7}$};
	\node at (-0.1,2/7) {\tiny $\frac{2}{7}$};
	\node at (-0.1,3/7) {\tiny $\frac{3}{7}$};
	\node at (-0.1,4/7) {\tiny $\frac{4}{7}$};
	\node at (-0.1,5/7) {\tiny $\frac{5}{7}$};
	\node at (-0.1,6/7) {\tiny $\frac{6}{7}$};
	\node at (-0.1,1) {\tiny $1$};
	\node at (70/169,-0.1) {\tiny $s_7$};
	\node at (99/169+0.05,-0.1) {\tiny $1-s_7$};
	\draw (0,0)--(0,1)--(1,1)--(1,0)--(0,0);
	\draw[dashed] (0,1/7)--(1,1/7);
	\draw[dashed] (0,2/7)--(1,2/7);
	\draw[dashed] (0,3/7)--(1,3/7);
	\draw[dashed] (0,4/7)--(1,4/7);
	\draw[dashed] (0,5/7)--(1,5/7);
	\draw[dashed] (0,6/7)--(1,6/7);
	\draw[thick] (0,0)--(2/169,2/7)--(3/169,1/7)--(5/169,3/7)--(7/169,1/7)--(9/169,3/7)--(10/169,2/7)--(12/169,4/7)--(14/169,2/7)--(15/169,3/7)--(17/169,1/7)--(19/169,3/7)--(20/169,2/7)--(22/169,4/7)--(24/169,2/7)--(26/169,4/7)--(27/169,3/7)--(29/169,5/7)--(31/169,3/7)--(32/169,4/7)--(34/169,2/7)--(36/169,4/7)--(38/169,2/7)--(39/169,3/7)--(41/169,1/7)--(43/169,3/7)--(44/169,2/7)--(46/169,4/7)--(48/169,2/7)--(50/169,4/7)--(52/169,3/7)--(53/169,5/7)--(55/169,3/7)--(56/169,4/7)--(58/169,2/7)--(60/169,4/7)--(61/169,3/7)--(63/169,5/7)--(65/169,3/7)--(67/169,5/7)--(68/169,4/7)--(70/169,6/7)--(72/169,4/7)--(73/169,5/7)--(75/169,3/7)--(77/169,5/7)--(79/169,3/7)--(80/169,4/7)--(82/169,2/7)--(84/169,4/7)--(85/169,3/7)--(87/169,5/7)--(89/169,3/7)--(90/169,4/7)--(92/169,2/7)--(94/169,4/7)--(96/169,2/7)--(97/169,3/7)--(99/169,1/7)--(101/169,3/7)--(102/169,2/7)--(104/169,4/7)--(106/169,2/7)--(108/169,4/7)--(109/169,3/7)--(111/169,5/7)--(113/169,3/7)--(114/169,4/7)--(116/169,2/7)--(118/169,4/7)--(121/169,3/7)--(121/169,5/7)--(123/169,3/7)--(125/169,5/7)--(126/169,4/7)--(128/169,6/7)--(130/169,4/7)--(131/169,5/7)--(133/169,3/7)--(135/169,5/7)--(137/169,3/7)--(138/169,4/7)--(140/169,2/7)--(142/169,4/7)--(143/169,3/7)--(145/169,5/7)--(147/169,3/7)--(149/169,5/7)--(151/169,4/7)--(152/169,6/7)--(154/169,4/7)--(155/169,5/7)--(157/169,3/7)--(159/169,5/7)--(160/169,4/7)--(162/169,6/7)--(164/169,4/7)--(166/169,6/7)--(167/169,5/7)--(1,1);
	\draw[dashed] (70/169,5/7)--(70/169,0);
	\draw[dashed] (99/169,1/7)--(99/169,0);
	\end{tikzpicture}
	\caption{Simple $n$-crooked maps $\sigma_n$ for $n=5,6$ and $7$. Note that $\sigma_5,\sigma_6$ and $\sigma_7$ are $3/5,3/6$ and $3/7$-crooked respectively.}\label{fig:sigma}
\end{figure} 

\begin{defn}
	For every integer $n\geq 3$ the maximal number of intervals of monotonicity of $\sigma_n$ of the same length is denoted by $\#_{\sigma_n}$.
\end{defn}

\begin{defn}
	We say that a function $f:[a,b]\to [a',b']$ where $[a,b],[a',b']\subset \mathbb{R}$ and $a<b, a'<b'$ is an \emph{odd function around the point $( \frac{a+b}{2},\frac{a'+b'}{2})\in \mathbb{R}^2$} if the graph of $f|_{\big[\frac{a+b}{2},b\big]}$ equals the rotation for angle $\pi$ of the graph of $f|_{\big[a,\frac{a+b}{2}\big]}$ around $(\frac{a+b}{2},\frac{a'+b'}{2})$.
\end{defn}

\begin{obs}\label{obs:sigma_n}
 	For every integer $n\geq 3$ 
 	\begin{enumerate}
 	\item $\sigma_n$ is a piecewise linear continuous function. 
 	\item it holds that $\#_{\sigma_n}=\scr[n]$. In particular, for even $n$, $\#_{\sigma_n}$ is even and it is odd for odd $n$.
 	\item  $\sigma_n$ has uniform slope being $\pm \frac{\scr[n]}{n}$.
 	\item \label{obs:sigma_n(4)} $\sigma_{n}$ is an odd function around the point $(1/2,1/2)${, i.e. it holds that $\sigma_{-n}^L(t)=\sigma_{n}^R(t+1/2)$ and $\sigma_{-n}^R(t+1/2)=\sigma_{n}^L(t)$ for all $t\in[0,1/2]$.}
	 	\end{enumerate}
\end{obs}

We will often use the following remark with $\eps=3/n$, however let us note that this estimate is far from optimal for $n$ small.

\begin{rem}\label{rem:crooked}
	By Proposition~3.5 in \cite{LM}, if $\eps>0$ and $n$ is sufficiently large to ensure $2/n<\eps$, the map $\sigma_n$ is $\eps$-crooked.
\end{rem}

\begin{rem}
If a map is $\eps$-crooked with small $\eps>0$ it cannot be a small perturbation of the identity map. To work with small perturbations of identity (which is necessity of Lemma~\ref{lem:MincUpdt})
we must give up crookedness over large subintervals (e.g. see Lemma~\ref{lem:MincUpdt}\eqref{Lcr:2}).
\end{rem}

In what follows we aim to define the maps $\lambda_{n,k}$ that we will work with throughout the section. It will be sufficient for our purposes to define this map for eventually every odd $n$; for even $n$ we could still construct a map $\lambda_{n,k}$ having all the important properties below but this requires a somewhat different construction and we therefore omit this part. Furthermore, it will be evident why we require $n\geq 7$ in the construction of $\lambda_{n,k}$ when we apply the results from \cite{MT}.

In what follows let us denote by 
\begin{equation}\label{eq:eta}
\eta:=\frac{\scr[n-1]}{2(\scr[n]+\scr[n-1])}.
\end{equation}

For each odd integer $n\geq 7$ and each integer $k\geq 1$ define the map
$$
\hat{\lambda}_{n,k} \colon [0,n+k-1]\to \Big[0,\frac{2n+k-2}{n}\Big]
$$
by the formula
\begin{equation}
\label{eq:lambdahat}
\hat{\lambda}_{n,k}(t):=\begin{cases}
\frac{n-1}{n}\sigma^R_{-(n-1)}(\frac{t-i}{2\eta}+\frac{1}{2})+\frac{i}{n},& \text{if } t\in[i,i+\eta],\\
\sigma_n((t-i-\eta)(\frac{1}{1-2\eta}))+\frac{i}{n},& \text{if } t\in[i+\eta,i+1-\eta],\\
\frac{n-1}{n}\sigma_{-(n-1)}^L(\frac{t-i-1}{2\eta}+\frac{1}{2})+\frac{i+1}{n},& \text{if } t\in[i+1-\eta,i+1].\\
\end{cases}
\end{equation}
 for some $i=\{0,1,\ldots, n+k-2\}$. See Figure~\ref{fig:hatlambda} for the graph of $\hat{\lambda}_{n,k}|_{[i,i+1]}$ when $n=7$.
 
 \begin{figure}[!ht]
 	\centering
 	\begin{tikzpicture}[scale=10]
 \draw[dashed] (0,0)--(1,1);
  \draw[dotted,orange,thick] (0,3/7)--(1,4/7);
 \node at (-0.1,0) {\small $0$};
 \node at (-0.1,1/7) {\small$1/7$};
 \node at (-0.1,2/7) {\small $2/7$};
 \node at (-0.1,3/7) {\small $3/7$};
 \node at (-0.1,4/7) {\small $4/7$};
 \node at (-0.1,5/7) {\small $5/7$};
 \node at (-0.1,6/7) {\small $6/7$};
 \node at (-0.1,1) {\small $1$};
 \draw[thick,blue] (0,0)--(0,1)--(1,1)--(1,0)--(0,0);
 \draw[dashed] (0,1/7)--(1,1/7);
 \draw[dashed] (0,2/7)--(1,2/7);
 \draw[dashed] (0,3/7)--(1,3/7);
 \draw[dashed] (0,4/7)--(1,4/7);
 \draw[dashed] (0,5/7)--(1,5/7);
 \draw[dashed] (0,6/7)--(1,6/7);
 \draw[thick](0,3/7)--(1/239,2/7)--(3/239,4/7)--(4/239,3/7)--(6/239,5/7)--(8/239,3/7)--(9/239,4/7)--(11/239,2/7)--(13/239,4/7)--(15/239,2/7)--(16/239,3/7)--(18/239,1/7)--(20/239,3/7)--(21/239,2/7)--(23/239,4/7)--(25/239,2/7)--(26/239,3/7)--(28/239,1/7)--(30/239,3/7)--(32/239,1/7)--(33/239,2/7)--(35/239,0)--(37/239,2/7)--(38/239,1/7)--(40/239,3/7)--(42/239,1/7)--(44/239,3/7)--(45/239,2/7)--(47/239,4/7)--(49/239,2/7)--(50/239,3/7)--(52/239,1/7)--(54/239,3/7)--(55/239,2/7)--(57/239,4/7)--(59/239,2/7)--(61/239,4/7)--(62/239,3/7)--(64/239,5/7)--(66/239,3/7)--(67/239,4/7)--(69/239,2/7)--(71/239,4/7)--(72/239,2/7)--(74/239,3/7)--(76/239,1/7)--(78/239,3/7)--(79/239,2/7)--(81/239,4/7)--(83/239,2/7)--(85/239,4/7)--(87/239,3/7)--(88/239,5/7)--(90/239,3/7)--(91/239,4/7)--(92/239,2/7)--(95/239,4/7)--(96/239,3/7)--(98/239,5/7)--(100/239,3/7)--(102/239,5/7)--(103/239,4/7)--(105/239,6/7)--(107/239,4/7)--(108/239,5/7)--(110/239,3/7)--(112/239,5/7)--(114/239,3/7)--(115/239,4/7)--(117/239,2/7)--(119/239,4/7)--(120/239,3/7)--(122/239,5/7)--(124/239,3/7)--(125/239,4/7)--(127/239,2/7)--(129/239,4/7)--(131/239,2/7)--(132/239,3/7)--(134/239,1/7)--(136/239,3/7)--(137/239,2/7)--(139/239,4/7)--(141/239,2/7)--(143/239,4/7)--(144/239,3/7)--(146/239,5/7)--(148/239,3/7)--(149/239,4/7)--(151/239,2/7)--(153/239,4/7)--(156/239,3/7)--(156/239,5/7)--(158/239,3/7)--(160/239,5/7)--(161/239,4/7)--(163/239,6/7)--(165/239,4/7)--(166/239,5/7)--(168/239,3/7)--(170/239,5/7)--(172/239,3/7)--(173/239,4/7)--(175/239,2/7)--(177/239,4/7)--(178/239,3/7)--(180/239,5/7)--(182/239,3/7)--(184/239,5/7)--(186/239,4/7)--(187/239,6/7)--(189/239,4/7)--(190/239,5/7)--(192/239,3/7)--(194/239,5/7)--(195/239,4/7)--(197/239,6/7)--(199/239,4/7)--(201/239,6/7)--(202/239,5/7)--(204/239,1)--(206/239,5/7)--(207/239,6/7)--(209/239,4/7)--(211/239,6/7)--(213/239,4/7)--(214/239,5/7)--(216/239,3/7)--(218/239,5/7)--(219/239,4/7)--(221/239,6/7)--(223/239,4/7)--(224/239,5/7)--(226/239,3/7)--(228/239,5/7)--(230/239,3/7)--(231/239,4/7)--(233/239,2/7)--(235/239,4/7)--(236/239,3/7)--(238/239,5/7)--(1,4/7);
 \draw[dashed] (35/239,0)--(35/239,1);
 \draw[dashed] (204/239,1)--(204/239,0);
 \node at (1/2,1.05) {$\sigma_7$};
 \node at (0.08,1.05) {$\sigma^R_{-6}$};
 \node at (0.92,1.05) {$\sigma^L_{-6}$};
  \node at (0,-0.05) {$i$};
 \node at (0.15,-0.05) {$i+\eta$};
 \node at (0.85,-0.05) {$i+1-\eta$};
  \node at (1,-0.05) {$i+1$};
 \node[circle,fill,orange, inner sep=1.5] at (0,3/7){};
 \node[circle,fill,orange, inner sep=1.5] at (1,4/7){};
 \node[circle,fill, inner sep=1.5] at (35/239,0){};
 \node[circle,fill, inner sep=1.5] at (204/239,1){};
  \node[circle,fill, inner sep=1.5] at (204/239,0){};
 \node[circle,fill, inner sep=1.5] at (0.5,0.5){};
  \node at (1.05,1/14) {$2$};
   \node at (1.05,3/14) {$18$};
    \node at (1.05,5/14) {$58$};
     \node at (1.05,7/14) {$83$};
      \node at (1.05,9/14) {$58$};
       \node at (1.05,11/14) {$18$};
        \node at (1.05,13/14) {$2$};
 	\end{tikzpicture}
 	\caption{Building blocks of the function $\hat{\lambda}_{n,k}$ for $n=7$. The numbers on the right side of the picture represent the number of intervals of monotonicity of diameter $1/n$ in the respective horizontal strip. Counting the numbers of such intervals will be important later when we will argue that maps $\lambda_{n,k}$ preserve Lebesgue measure. The dotted line represents the diagonal for the map $\lambda_{n,k}$.}\label{fig:hatlambda}
 \end{figure}
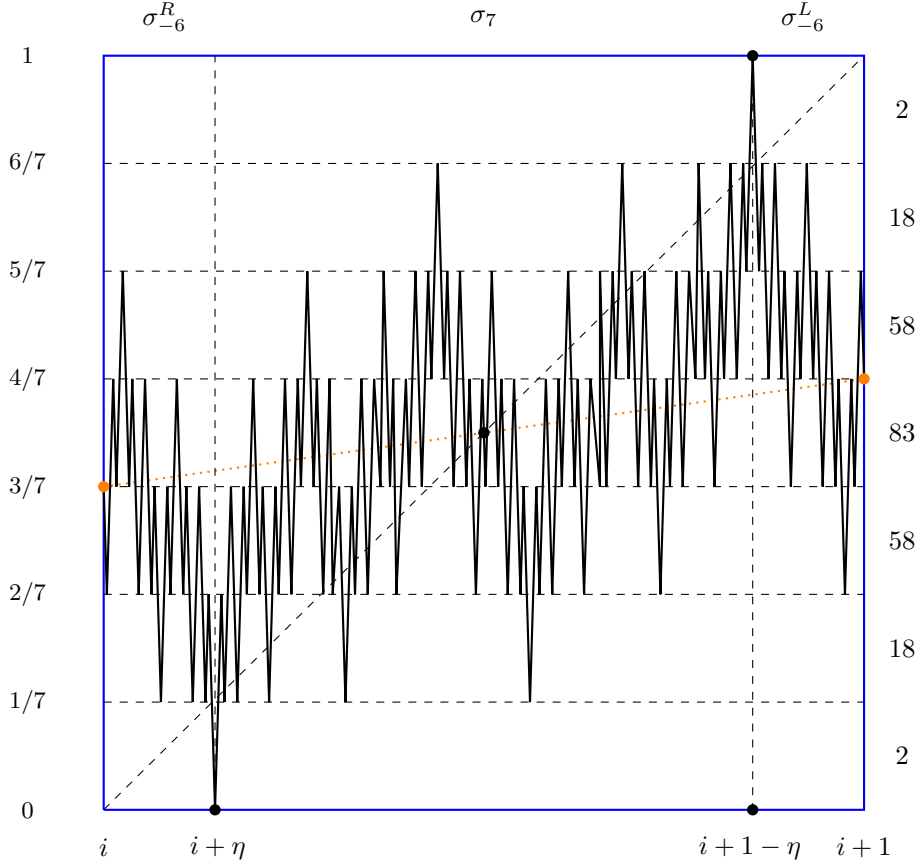 
 
 \begin{obs}\label{obs:lambdahat}
 For each odd integer $n\geq 7$ and each $k\in \mathbb{N}$ the map $\hat{\lambda}_{n,k}$ is 
 \begin{enumerate}
  \item a continuous and piecewise linear function with the uniform slope $\frac{\scr[n]+\scr[n-1]}{n}$.
  \item \label{obs:lambdahat2} an odd function around the point $\big(\frac{n+k-1}{2},\frac{2n+k-2}{2n}\big)$.
  \end{enumerate}
 \end{obs}

\begin{lem}\label{lem:hatcrooked}
	For every odd integer $n\geq 7$ and every integer $k\geq 1$ it holds that if $t,s\in [0,n+k-1]$ are such that $|\hat{\lambda}_{n,k}(t)-\hat{\lambda}_{n,k}(s)|<\frac{n-1}{n}$ then $\hat{\lambda}_{n,k}$ is $\frac{3}{n}$-crooked between $\hat{\lambda}_{n,k}(t)$ and $\hat{\lambda}_{n,k}(s)$. 
\end{lem}

\begin{proof}
First note that the function $\hat{\lambda}_{n,k}$ is generated using rescaled $\sigma_{-(n-1)}^R$, $\sigma_n$ and $\sigma_{-(n-1)}^L$ that are properly shifted vertically; to simplify the notation in this proof we will refer to the three parts of the definition of $\hat \lambda_{n,k}$ simply by $\sigma_{-(n-1)}^R$, $\sigma_n$ and $\sigma_{-(n-1)}^L$ while remembering about the rescaling and shift. 
Therefore, $\hat{\lambda}_{n,k}$ consists of blocks of the form 
$$\sigma_{-(n-1)}^R\sigma_n\sigma_{-(n-1)}^L\sigma_{-(n-1)}^R\sigma_n\sigma_{-(n-1)}^L\sigma_{-(n-1)}^R\sigma_n \ldots \sigma_{-(n-1)}^L\sigma_{-(n-1)}^R\sigma_n\sigma_{-(n-1)}^L$$
which simplifies to
$$\sigma_{-(n-1)}^R\sigma_n \sigma_{-(n-1)} \sigma_n \sigma_{-(n-1)}\sigma_n\ldots \sigma_{-(n-1)}\sigma_n\sigma_{-(n-1)}^L,$$
see Figure~\ref{fig:hatlambda}.
Let us denote the domains of the latter blocks by $S_i$ for $i\in \{1,\ldots, 2(n+k)\}$ and which gives a well defined order on them.\\
Fix $a,b\in [0,\frac{2n+k-2}{n}]$, so that $|a-b|<\frac{n-1}{n}$. If points $c,d\in [0,n+k-1]$, where $f(c)=a$ and $f(d)=b$, are contained in some $S_i$ for $i\in  \{1,\ldots, 2(n+k)\}$, the claim follows directly from Remark~\ref{rem:crooked} (when the points are from either $S_1$ or $S_{2(n+k)}$ note that the images are restrictions of $\sigma_{-(n-1)}$ and thus we can again use Remark~\ref{rem:crooked}).\\
Now assume that points $c,d\in [0,n+k-1]$ are contained in two adjacent blocks $S_j<S_{j+1}$, $\hat \lambda_{n,k}|_{S_j}=\sigma_{-(n-1)}$ and $\hat \lambda_{n,k}|_{S_{j+1}}=\sigma_{n}$ for some $j\in \{1,\ldots, 2(n+k)-1\}$. Assume that $c\in S_j$ and $d\in S_{j+1}$ (case when $c\in S_{j+1}$ and $d\in S_{j}$ is proven analogously). Then, since $\hat{\lambda}_{n,k}(S_j)\subset \hat{\lambda}_{n,k}(S_{j+1})$ it follows that there always exist $c'\in S_{j+1}$ such that $c<c'<d$ and so that $f(c')=a$. Thus we obtain the claim using Remark~\ref{rem:crooked} again.\\ 
Now let us assume that $c,d\in  [0,n+k-1]$ are contained in two non-adjacent blocks, $c\in S_j$ and  $d\in S_{j'}$ for $|j-j'|\geq 2$. If $a\in \hat{\lambda}_{n,k}(S_{j'})$ or $b\in \hat{\lambda}_{n,k}(S_{j})$ then we can find two adjacent blocks between $S_j$ and $S_{j'}$ and use the arguments from the preceding paragraph as we assumed that $|a-b|<\frac{n-1}{n}$. 
Note that $a\in \hat{\lambda}_{n,k}(S_{j'})$ or $b\in \hat{\lambda}_{n,k}(S_{j})$ holds always except if $\hat{\lambda}_{n,k}|_{S_j}=\hat{\lambda}_{n,k}|_{S_j'}=\sigma_n$, $|j-j'|= 2$ and $\frac{n-2}{n}\leq|a-b|\leq \frac{n-1}{n}$. 
But in this case observe that $\frac{3}{n}$-crookedness of $\sigma_n$ assures that we can find a point between $c$ and $d$ in either $S_j$ or $S_{j'}$ with the required image value.
\end{proof}

\begin{defn}\label{def:flip}
For all $n,k\in \mathbb{N}$ define the {\em flip} map 
$$\Fl\colon \bigg[-\frac{n-1}{2(n+k-1)},\frac{3n+2k-3}{2(n+k-1)}\bigg]\to I$$ by

\begin{equation}
\Fl(s):=\begin{cases}
-s &, \text{if } s \leq 0,\\
s &, \text{if } s\in I,\\
2-s &, \text{if } s\geq 1.
\end{cases}
\end{equation}
\end{defn}
 
 Now we have all the ingredients to define the final map with which we will work in this section.
 
 \begin{defn}\label{def:lambda}
 	For every odd integer $n\geq 7$ and every $k\in\mathbb{N}$ define the map
 	$\lambda_{n,k} \colon I\to I$ by 
 	\begin{equation}
 	\label{eq:lambda}
 	\lambda_{n,k}(t):=\Fl\bigg(\frac{n}{n+k-1}\hat{\lambda}_{n,k}(t(n+k-1))-\frac{n-1}{2(n+k-1)}\bigg)
 	\end{equation}
 \end{defn}
for all $t\in I$.

See Figure~\ref{fig:lambda} for schematic picture of $\lambda_{n,k}$ for $n=7$ and note that the properly scaled ''middle'' building blocks of $\lambda_{n,k}$ are as on Figure~\ref{fig:hatlambda}.

\begin{obs}\label{obs:lambda1}
For each odd integer $n\geq 7$ and each integer $k\geq 1$ the map $\lambda_{n,k}$ is 
\begin{enumerate}
	\item \label{obs:lambda1(1)} a continuous piecewise linear function with the uniform slope being $\pm(\scr[n]+\scr[n-1])$,
	\item \label{obs:lambda1(2)} an odd function around the point $(1/2,1/2)$,
	\item \label{obs:lambda1(3)} such that $\lambda_{n,k}\big(\frac{j}{n+k-1}\big)=\frac{j}{n+k-1}$ for all $j\in \{0,\ldots, n+k-1\}$.
\end{enumerate}
\end{obs}

Now we will turn to the proof that the function $\lambda_{n,k}$ preserves Lebesgue measure for all odd $n\geq 7$ and $k\geq 1$.

\begin{defn}
	For every $j\in\{1,\ldots, n+k-1\}$ and some fixed integer $k\geq 1$ and odd integer $n\geq 7$ denote by $I_j:=\frac{1}{n+k-1}[j-1,j]\subset I$ and let
	$$
	V_j:=I_j\times I 
	$$
	denote the {\em $j$-th vertical strip} and
	$$
	H_j:=I\times I_j
	$$
	denote the {\em $j$-th horizontal strip}. Let $p_1:I\times I\to I$ ($p_2:I\times I\to I$) be the natural projection onto the first (second) coordinate.
	Furthermore, define the maximal number of intervals of monotonicity of $\lambda_{n,k}$ of the diameter exactly $\frac{1}{n+k-1}$ in $V_j$ (resp. $H_j$) by 
	$$
	\#(V_j) \text{ (resp. } \#(H_j)).
 	$$
\end{defn}

\begin{obs}\label{obs:lambda2}
	 For each odd integer $n\geq 7$ and each integer $k\geq 1$ it holds that
\begin{enumerate}
	\item \label{obs:lambda2(1)} the function $\lambda_{n,k}|_{I_j}$ is an odd function around the point $\Big(\frac{2j-1}{2(n+k-1)},\frac{2j-1}{2(n+k-1)}\Big)$ for any $j\in \{\frac{n+1}{2},\ldots, k+\frac{n-1}{2}\}$. 
	\item \label{obs:lambda2(3)}  $\#(V_j)=\scr[n]+\scr[n-1]$ for any $j\in  \{1,\ldots, n+k-1\}$ (since function $\Fl$ does not change $\#(V_j)$).
	\item \label{obs:lambda2(4)} $\diam(\lambda_{n,k}({I_j}))=\frac{n}{n+k-1}$ if $j\in \{\frac{n+1}{2},\ldots, k+\frac{n-1}{2}\}$. 
	\item \label{obs:lambda2(5)} $\diam(\lambda_{n,k}({I_j}))=\frac{n+2j-1}{2(n+k-1)}$ if $j\in \{1,\ldots, \frac{n-1}{2}\}$.
\end{enumerate}
\end{obs}

To check whether $\lambda_{n,k}$ preserves Lebesgue measure we will also implicitly use the following simple observation.

\begin{obs}\label{l:3} A piecewise monotone and up to a finite points $E$ differentiable interval map $f$ preserves Lebesgue measure if and only if 
   \begin{equation}\label{e:3}
   \forall~y\in I\setminus f(E)\colon~\sum_{x\in  f^{-1}(y)}\frac{1}{\vert f'(x)\vert}=1.
 \end{equation}
 In other words, for any map $f\in C_{\lambda}(I)$ and any nondegenerate interval $J\subset I$,
\begin{equation}\label{eq:basic}
\sum_{K\in\mathrm{Comp}(f^{-1}(J))}\frac{\lambda(K)}{\lambda(J)}=1,
\end{equation}
where $\mathrm{Comp}(f^{-1}(J))$ denotes the set of all connected components of $f^{-1}(J)$.
 \end{obs}

\begin{prop}\label{prop:LebPres}
	For each odd integer $n\geq 7$ and each integer $k\geq 1$ the map $\lambda_{n,k}$ preserves Lebesgue measure on $I$.
\end{prop}

\begin{figure}[!ht]
	\centering
	\begin{tikzpicture}[scale=0.8]
	\draw[thick] (0,0)--(11,0)--(11,11)--(0,11)--(0,0);
	\draw [thick,red] (0,0)--(0,-3)--(1,-3)--(1,-2)--(2,-2)--(2,-1)--(3,-1)--(3,0);
	\draw [thick,red] (0,0)--(0,3)--(1,3)--(1,2)--(2,2)--(2,1)--(3,1)--(3,0);
	\draw [thick,blue] (5,9)--(6,9)--(6,2)--(5,2)--(5,9);
	\draw [thick,blue] (9,5)--(9,6)--(2,6)--(2,5)--(9,5);
	\draw[dashed,red] (0,-1)--(2,-1)--(2,0);
	\draw[dashed,red] (0,-2)--(1,-2)--(1,0);
	\draw[dashed] (0,1)--(11,1);
		\draw[dashed] (0,2)--(11,2);
		\draw[dashed] (0,3)--(11,3);
		\draw[dashed] (0,4)--(11,4);
		\draw[dashed] (0,5)--(11,5);
		\draw[dashed] (0,6)--(11,6);
		\draw[dashed] (0,7)--(11,7);
	\draw[dashed] (0,8)--(11,8);
	\draw[dashed] (0,9)--(11,9);
	\draw[dashed] (0,10)--(11,10);
 \draw[dashed] (1,0)--(1,11);
 \draw[dashed] (2,0)--(2,11);
 \draw[dashed] (3,0)--(3,11);
 \draw[dashed] (4,0)--(4,11);
 \draw[dashed] (5,0)--(5,11);
 \draw[dashed] (6,0)--(6,11);
 \draw[dashed] (7,0)--(7,11);
 \draw[dashed] (8,0)--(8,11);
 \draw[dashed] (9,0)--(9,11);
 \draw[dashed] (10,0)--(10,11);
 \draw[dotted,thick,orange] (0,0)--(11,11);
 \node at (0.5,0.5) {$141$};
 \node at (1.5,1.5) {$85$};
 \node at (2.5,2.5) {$83$};
 \node at (3.5,3.5) {$83$};
  \node at (4.5,4.5) {$83$};
   \node at (5.5,5.5) {$83$};
    \node at (6.5,6.5) {$83$};
     \node at (7.5,7.5) {$83$};
      \node at (8.5,8.5) {$83$};
       \node at (9.5,9.5) {$85$};
        \node at (10.5,10.5) {$141$};
 \node at (0.5,1.5) {$76$};
 \node at (1.5,2.5) {$58$};
 \node at (2.5,3.5) {$58$};
 \node at (3.5,4.5) {$58$};  
  \node at (4.5,5.5) {$58$}; 
   \node at (5.5,6.5) {$58$}; 
    \node at (6.5,7.5) {$58$}; 
     \node at (7.5,8.5) {$58$}; 
      \node at (8.5,9.5) {$58$}; 
       \node at (9.5,10.5) {$76$};
  \node at (0.5,2.5) {$20$};         
   \node at (1.5,3.5) {$18$};
     \node at (2.5,4.5) {$18$};    
       \node at (3.5,5.5) {$18$};    
         \node at (4.5,6.5) {$18$};    
           \node at (5.5,7.5) {$18$};    
             \node at (6.5,8.5) {$18$};
              \node at (7.5,9.5) {$18$};
               \node at (8.5,10.5) {$20$}; 
  \node at (0.5,3.5) {$2$};
   \node at (1.5,4.5) {$2$};
    \node at (2.5,5.5) {$2$};
     \node at (3.5,6.5) {$2$};
      \node at (4.5,7.5) {$2$};
       \node at (5.5,8.5) {$2$};
        \node at (6.5,9.5) {$2$};
         \node at (7.5,10.5) {$2$};  
  \node at (1.5,0.5) {$76$};
 \node at (2.5,1.5) {$58$};
 \node at (3.5,2.5) {$58$};
 \node at (4.5,3.5) {$58$};  
 \node at (5.5,4.5) {$58$}; 
 \node at (6.5,5.5) {$58$}; 
 \node at (7.5,6.5) {$58$}; 
 \node at (8.5,7.5) {$58$}; 
 \node at (9.5,8.5) {$58$}; 
 \node at (10.5,9.5) {$76$};
 \node at (2.5,0.5) {$20$};         
 \node at (3.5,1.5) {$18$};
 \node at (4.5,2.5) {$18$};    
 \node at (5.5,3.5) {$18$};    
 \node at (6.5,4.5) {$18$};    
 \node at (7.5,5.5) {$18$};    
 \node at (8.5,6.5) {$18$};
 \node at (9.5,7.5) {$18$};
 \node at (10.5,8.5) {$20$}; 
 \node at (3.5,0.5) {$2$};
 \node at (4.5,1.5) {$2$};
 \node at (5.5,2.5) {$2$};
 \node at (6.5,3.5) {$2$};
 \node at (7.5,4.5) {$2$};
 \node at (8.5,5.5) {$2$};
 \node at (9.5,6.5) {$2$};
 \node at (10.5,7.5) {$2$};  
              
 \node at (11.5,0.5) {$239$};
  \node at (11.5,1.5) {$239$};
   \node at (11.5,2.5) {$239$};
    \node at (11.5,3.5) {$239$};
     \node at (11.5,4.5) {$239$};
      \node at (11.5,5.5) {$239$};
       \node at (11.5,6.5) {$239$};
        \node at (11.5,7.5) {$239$};
         \node at (11.5,8.5) {$239$};
          \node at (11.5,9.5) {$239$};
           \node at (11.5,10.5) {$239$};
            \node at (0.5,11.5) {$239$};
             \node at (1.5,11.5) {$239$};
              \node at (2.5,11.5) {$239$};
               \node at (3.5,11.5) {$239$};
                \node at (4.5,11.5) {$239$};
                 \node at (5.5,11.5) {$239$};
                  \node at (6.5,11.5) {$239$};
                   \node at (7.5,11.5) {$239$};
                    \node at (8.5,11.5) {$239$};
                     \node at (9.5,11.5) {$239$};
                      \node at (10.5,11.5) {$239$};
  \node at (0.5,-0.5) {\color{red}$58$};
 \node at (0.5,-1.5) {\color{red}$18$};
 \node at (0.5,-2.5) {\color{red}$2$};
 \node at (1.5,-0.5) {\color{red}$18$};
 \node at (1.5,-1.5) {\color{red}$2$};
 \node at (2.5,-0.5) {\color{red}$2$};
\draw[thick] (-0.5,-0.5) arc(-90:-270:0.5);
\draw[thick] (-0.5,-1.5) arc(-90:-270:1.5);
\draw[thick](-0.5,-2.5) arc(-90:-270:2.5);
\draw[thick] (-0.2,-0.5)--(-0.5,-0.5);
\draw[thick] (-0.2,-1.5)--(-0.5,-1.5);
\draw[thick] (-0.2,-2.5)--(-0.5,-2.5);
\draw[thick,->] (-0.5,0.5)--(-0.2,0.5);
\draw[thick,->] (-0.5,1.5)--(-0.2,1.5);
\draw[thick,->] (-0.5,2.5)--(-0.2,2.5);
 \node[circle,fill,orange, inner sep=1.5] at (0,0){};
\node[circle,fill,orange, inner sep=1.5] at (1,1){};
\node[circle,fill,orange, inner sep=1.5] at (2,2){};
\node[circle,fill,orange, inner sep=1.5] at (3,3){};
\node[circle,fill,orange, inner sep=1.5] at (4,4){};
\node[circle,fill,orange, inner sep=1.5] at (5,5){};
\node[circle,fill,orange, inner sep=1.5] at (6,6){};
\node[circle,fill,orange, inner sep=1.5] at (7,7){};
\node[circle,fill,orange, inner sep=1.5] at (8,8){};
\node[circle,fill,orange, inner sep=1.5] at (9,9){};
\node[circle,fill,orange, inner sep=1.5] at (10,10){};
\node[circle,fill,orange, inner sep=1.5] at (11,11){};
	\end{tikzpicture}
	\caption{In this picture we sum up the maximal number of injective branches of $\lambda_{n,k}$ of diameter $\frac{1}{\scr[n]+\scr[n-1]}$ where $n=7$ and $k=5$ 
	(note also that boxes divide the interval in $n+k-1=11$ pieces) in each box $B_{j,l}$ (the number in the boxes represents $\#(B_{j,l})$). The numbers above and on the right hand side of the large black square represent $\#(V_j)$ and $\#(H_l)$ respectively. The vertical blue rectangle denotes the position of a rescaled version of the map $\hat \lambda_{7,k}$ from Figure~\ref{fig:hatlambda} (its copies are placed all along the diagonal). Note that by the first case of Equation~\ref{eq:lambda} (using the flip function), $\#(B_{1,1})=141=83+58$, $\#(B_{2,2})=85=83+2$, $\#(B_{1,2})=\#(B_{2,1})=76=58+18$ and  $\#(B_{1,3})=\#(B_{3,1})=20=18+2$. Below the main black square the boxes (drawn in red) are the boxes that are flipped up (first case in Definition~\ref{def:flip}); the numbers inside represent the maximal number of injective branches of diameter $\frac{1}{239}$. Analogous procedure is done in the upper right corner above the large black square, however we omit drawing that part due to symmetry of $\lambda_{n,k}$.}\label{fig:lambda}
\end{figure}

\begin{proof}
From Observation~\ref{obs:lambda2} (\ref{obs:lambda2(3)}) it holds that $\#(V_j)=\scr[n]+\scr[n-1]$ for all $j\in \{1,\ldots, n+k-1\}$ and since by Observation~\ref{obs:lambda1} (\ref{obs:lambda1(1)}) $\lambda_{n,k}$ has uniform slope (in the absolute value) we only need to show that $\#(H_j)=\scr[n]+\scr[n-1]$ since $\diam(p_1(V_{j}))=\diam(p_2(H_{j'}))$ for any $j,j'\in \{1,\ldots, n+k-1\}$.
We will consider {\em boxes} $B_{j,l}:=V_j\cap H_l$ and the number of injective branches of $\lambda_{n,k}$ of $B_{j,l}$ denoted by $\#(B_{j,l})$  for all $j,l\in\{1,\ldots, n+k-1\}$, see Figure~\ref{fig:lambda}.\\
First assume that $j\in \{\frac{n+1}{2},\ldots, k+\frac{n-1}{2}\}$. Then $\#(V_j)=\#(B_{j,j-\frac{n-1}{2}})+\ldots+\#(B_{j,j})+\ldots +\#(B_{j,j+\frac{n-1}{2}})=\#(B_{j,j})+2\#(B_{j+1,j})+2\#(B_{j+2,j})+ \ldots +2\#(B_{j+\frac{n-1}{2},j})${, due to Observation~\ref{obs:lambda2} (\ref{obs:lambda2(1)}) and since $n$ is odd}. But note that (see the highlighted middle part of Figure~\ref{fig:lambda}) $\#(B_{j,m})=\#(B_{m,j})$ for all $m\in \{j-\frac{n-1}{2},\ldots ,j+\frac{n-1}{2}\}$ and thus $\#(H_j)=\#(B_{j-\frac{n-1}{2},j})+\ldots+\#(B_{j,j})+\ldots +\#(B_{j+\frac{n-1}{2},j})=\#(B_{j,j-\frac{n-1}{2}})+\ldots+\#(B_{j,j})+\ldots +\#(B_{j,j+\frac{n-1}{2}})=\#(V_j)$ which finishes this part of the proof.\\
Now assume that $j\in \{1,\ldots, \frac{n-1}{2}\} \cup \{k+\frac{n+1}{2},\ldots, n+k-1\}$. By Observation~\ref{obs:lambda1} (\ref{obs:lambda1(2)}) it is enough to check the claim for $j\in \{1,\ldots, \frac{n+1}{2}\}$. By Definition~\ref{def:flip} it holds that $\#(H_j)=\#(B_{1,j})+\#(B_{2,j})\ldots+\#(B_{j,j})+\ldots +\#(B_{j+\frac{n-1}{2},j})=\#(B_{j,j})+2\#(B_{j+1,j})+2\#(B_{j+2,j})+\ldots +2\#(B_{j+\frac{n-1}{2},j})=\#(V_j)$ (see the lower left corner of Figure~\ref{fig:lambda}){, which finishes this part of proof by Observation~\ref{obs:lambda2} (\ref{obs:lambda2(3)})}. Note that the last argument crucially depends on the choice of the flip function $\Fl$.\\
Thus for every nondegenerate interval $J\subset I$ we can use Equation~\eqref{eq:basic} and therefore $\lambda_{n,k}\in C_{\lambda}(I)$.
\end{proof}

Now we will prove that $\lambda_{n,k}$ fits in the context of Proposition 5 of \cite{MT} (there such a perturbation map is denoted by $g$).  Note that the map $g$ constructed in that proposition does not fit our purposes here since it does not preserve Lebesgue measure, it does not have uniform slope (in the absolute value) and furthermore it is not an odd function around $(1/2,1/2)$, which we use in the subsequent arguments heavily (see Figure~\ref{fig:MTmap} for a map that captures the essence of construction of map $g$ from \cite{MT}).

\begin{figure}[!ht]
	\centering
	\begin{tikzpicture}[scale=4]
	\draw (0,0)--(0,1)--(1,1)--(1,0)--(0,0);
	\draw[dashed] (0,0)--(1,1);
	\draw[dashed] (0,1/5)--(1,1/5);
	\draw[dashed] (0,2/5)--(1,2/5);
	\draw[dashed] (0,3/5)--(1,3/5);
	\draw[dashed] (0,4/5)--(1,4/5);
	\draw[dashed] (1/5,0)--(1/5,1);
	\draw[dashed] (2/5,0)--(2/5,1);
	\draw[dashed] (3/5,0)--(3/5,1);
	\draw[dashed] (4/5,0)--(4/5,1);
	\node at (-0.1,0) {$0$};
	\node at (-0.1,1/5) { $\frac{1}{5}$};
	\node at (-0.1,2/5) { $\frac{2}{5}$};
	\node at (-0.1,3/5) { $\frac{3}{5}$};
	\node at (-0.1,4/5) { $\frac{4}{5}$};
	\node at (1/5,-0.1) { $\frac{1}{5}$};
	\node at (2/5,-0.1) { $\frac{2}{5}$};
	\node at (3/5,-0.1) { $\frac{3}{5}$};
	\node at (4/5,-0.1) { $\frac{4}{5}$};
	\node at (1,-0.1) { $1$};
	\node at (-0.1,1) { $1$};
	\draw[thick]
	(0,0)--(2/50,2/5)--(3/50,1/5)--(5/50,3/5)--(7/50,1/5)--(8/50,2/5)--(1/5,0)--(26/120,2/5)--(27/120,1/5)--(29/120,3/5)--(31/120,1/5)--(33/120,3/5)--(34/120,2/5)--(36/120,4/5)--(38/120,2/5)--(39/120,3/5)--(41/120,1/5)--(43/120,3/5)--(45/120,1/5)--(46/120,2/5)--(2/5,0)--(118/290,2/5)--(119/290,1/5)--(121/290,3/5)--(123/290,1/5)--(125/290,3/5)--(126/290,2/5)--(128/290,4/5)--(130/290,2/5)--(131/290,3/5)--(133/290,1/5)--(135/290,3/5)--(136/290,2/5)--(138/290,4/5)--(140/290,2/5)--(142/290,4/5)--(143/290,3/5)--(1/2,1)--(147/290,3/5)--(148/290,4/5)--(150/290,2/5)--(152/290,4/5)--(154/290,2/5)--(155/290,3/5)--(157/290,1/5)--(159/290,3/5)--(160/290,2/5)--(162/290,4/5)--(164/290,2/5)--(165/290,3/5)--(167/290,1/5)--(169/290,3/5)--(171/290,1/5)--(172/290,2/5)--(3/5,0)--(176/290,2/5)--(177/290,1/5)--(179/290,3/5)--(181/290,1/5)--(183/290,3/5)--(184/290,2/5)--(186/290,4/5)--(188/290,2/5)--(189/290,3/5)--(191/290,1/5)--(193/290,3/5)--(194/290,2/5)--(196/290,4/5)--(198/290,2/5)--(200/290,4/5)--(201/290,3/5)--(7/10,1)--(86/120,3/5)--(87/120,4/5)--(89/120,2/5)--(91/120,4/5)--(93/120,2/5)--(94/120,3/5)--(4/5,1/5)--(98/120,3/5)--(99/120,2/5)--(101/120,4/5)--(103/120,2/5)--(105/120,4/5)--(106/120,3/5)--(108/120,1)--(45/50,1)--(47/50,3/5)--(48/50,4/5)--(1,2/5);
	\end{tikzpicture}
	\caption{This map captures the essence of the construction of map $g$ from Proposition 5 of \cite{MT}.}\label{fig:MTmap}
\end{figure}
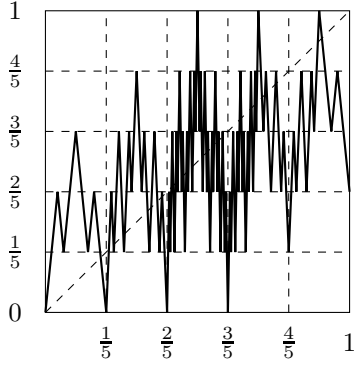 

\begin{obs}\label{obs:gammaapart}
Let $x_j\in I_j$ be the minimal number such that
 $$\lambda_{n,k}(x_j)=\max\{\lambda_{n,k}(x); x\in I_j\}$$  for all  $j\in \{1,\ldots,k+\frac{n-1}{2}\}$. 
Then
  $|x_{j-1}-x_{j}|=\frac{1}{n+k-1}$  and $|\lambda_{n,k}(x_{j-1})-\lambda_{n,k}(x_{j})|=\frac{1}{n+k-1}$ for all $j\in \{1,\ldots, k+\frac{n-1}{2}\}$.
\end{obs}

{
The result below is an analogue of Proposition 5 in \cite{MT}. Note that in conditions \eqref{Lcr:1} and \eqref{Lcr:3} below we can put larger number $3\gamma$ in place of $\gamma$ and the conclusions still hold. Therefore, our statements are exact analogues of Proposition 5 in \cite{MT}, with $\gamma$ there playing the role of $3\gamma$ in Lemma~\ref{lem:MincUpdt}.}
\begin{lem}\label{lem:MincUpdt}
	Let $\lambda_{n,k}$ be defined as in Definition~\ref{def:lambda}. Set $\eps:=\frac{n-1}{n+k-1}$, $\gamma:=\frac{1}{n+k-1}$. Then the following statements hold  for every odd integer $n\geq 7$ and $k\geq 1$:
	\begin{enumerate}[(i)]
		\item\label{Lcr:1} $\rho(\lambda_{n,k},\mathrm{id})<\eps/2+\gamma$,
		\item\label{Lcr:2} for every $a$ and $b$ such that $|a - b| < \eps$, $\lambda_{n,k}$ is $3\gamma$-crooked between $a$ and $b$,
		\item\label{Lcr:3} for each subinterval $A$ of $I$ we have $\diam(\lambda_{n,k}(A)) \geq \diam(A)$,
	and if, additionally, $\diam(A)> \gamma$, then
	\begin{enumerate}[(a)]
		\item\label{Lcr:4} $\diam(\lambda_{n,k}(A)) > \eps/2$, 
		\item\label{Lcr:5} $A \subset \lambda_{n,k}(A)$ and
		\item\label{Lcr:6} $\lambda_{n,k}(B) \subset B(\lambda_{n,k}(A),r + \gamma)$ for each non-negative real number $r$ and each set $B \subset B(A,r)$.
	\end{enumerate}	
	\end{enumerate}
\end{lem}

\begin{proof}
\eqref{Lcr:1} From the construction of the map $\lambda_{n,k}$ it follows that $$|t-\lambda_{n,k}(t)|<\frac{n+1}{2}\frac{1}{n+k-1}=\frac{n-1}{2(n+k-1)}+\frac{1}{n+k-1}=\eps/2+\gamma.$$

\eqref{Lcr:2} Now let us prove that for every $a,b\in I$ so that $|a-b|<\eps$ it follows that $\lambda_{n,k}$ is $3\gamma$-crooked between $a$ and $b$. First let us consider the map $\tilde{\lambda}_{n,k}: I\to [-\frac{n-1}{2(n+k-1)},1+\frac{n-1}{2(n+k-1)}]$ defined for every odd integer $n\geq 7$ and every integer $k\geq 1$ by
$$
\tilde{\lambda}_{n,k}(t):=\frac{n}{n+k-1}\hat{\lambda}_{n,k}(t(n+k-1))-\frac{n-1}{2(n+k-1)} \text{ for } t\in I.
$$
Note that one obtains $\lambda_{n,k}$ from $\tilde{\lambda}_{n,k}$ applying the flip function from Definition~\ref{def:flip}. Applying proper rescaling factor it follows from Lemma~\ref{lem:hatcrooked} that for every $a,b\in [-\frac{n-1}{ 2(n+k-1)},1+\frac{n-1}{2(n+k-1)}]$ so that $|a-b|<\eps$ we have that the map $\tilde{\lambda}_{n,k}$ is $3\gamma$-crooked between $a$ and $b$. However, since one obtains $\lambda_{n,k}$ from $\tilde{\lambda}_{n,k}$ and the flip function by the definition at most decreases the distances between the function values, the claim follows immediately. 

\eqref{Lcr:3} If $A\subset I$ is a subinterval so that $\diam(A)< \frac{2}{(\scr[n]+\scr[n-1])(n+k-1)}$, then by Observation~\ref{obs:lambda1} (\ref{obs:lambda1(1)}) it holds that $\diam(\lambda_{n,k}(A))> (\scr[n]+\scr[n-1])\frac{\diam(A)}{2}>\diam(A)$. If a subinterval $A\subset I$ is such that $\gamma>\diam(A)\geq \frac{2}{(\scr[n]+\scr[n-1])(n+k-1)}$, then it follows that $\diam(\lambda_{n,k}(A))\geq \gamma>\diam(A)$ because $A$ contains at least one full interval of monotonicity of the diameter of image being $\gamma$. Now assume that $\gamma\leq \diam (A)\leq 2\gamma$.  Then there are $x,y\in A$ such that $\lambda_{n,k}(y)=\max \lambda_{n,k} (I_j)$ and $\lambda_{n,k}(x)=\min \lambda_{n,k} (I_{j'})$ where $|j-j'|\leq 1$ and $A\cap I_j\neq \emptyset$,  $A\cap I_{j'}\neq \emptyset$. Note that $A$ is contained in at most three different intervals $I_i$ and $\lambda_{n,k}(I_i)$ covers itself and two neighbouring intervals on each side, provided it is not interval containing endpoints $0$ or $1$ (and thus such neighbouring intervals indeed exist, see Figure~\ref{fig:lambda}). This means that $A\subset [\lambda_{n,k}(x),\lambda_{n,k}(y)]\subset \lambda_{n,k}(A)$, in particular $\diam(\lambda_{n,k}(A))\geq \diam(A)$. If $A$ contains $0$ or $1$ it follows from the construction of $\lambda_{n,k}$ that $A\subset \lambda_{n,k}(A)$ and thus subsequently $\diam(\lambda_{n,k}(A))\geq \diam(A)$; even more, $\diam(\lambda_{n,k}(A))\geq \frac{n+1}{2(n+k-1)}$.
When $\diam(A)>2\gamma$, then there are $j<j'$ such that $I_j\cup \ldots \cup I_{j'}\subset A$ and $\diam(I_j\cup \ldots \cup I_{j'})$ is maximal possible for the interval $I_j\cup \ldots \cup I_{j'}$ under inclusion (meaning one cannot take smaller $j$ or larger $j'$). But then clearly $A\subset [\min\lambda_{n,k}(I_j), \max \lambda_{n,k}(I_{j'})]$ (see Figure~\ref{fig:lambda}) which completes the proof. The claim \eqref{Lcr:5} is a consequence of the proof of \eqref{Lcr:3}.  

\eqref{Lcr:4} This part follows from the arguments in the last paragraph if one notes that $\eps/2=\frac{n-1}{2(n+k-1)}$. Namely, $|\tilde{\lambda}_{n,k}(x)-\tilde{\lambda}_{n,k}(y)|\geq\frac{n+1}{n+k-1}$ where $x$ and $y$ are as in the previous paragraph. After applying the flip function we obtain $|\lambda_{n,k}(x)-\lambda_{n,k}(y)|\geq\frac{n+1}{2(n+k-1)}$, see Figure~\ref{fig:lambda}.

\eqref{Lcr:6} For this part Observation~\ref{obs:gammaapart} is crucial. Namely, in intervals $I_j$ the first maximal value of $\lambda_{n,k}$ lies $\gamma$ apart for all $j\in\{1,\ldots, k+\frac{n-1}{2}\}$  and it is exactly $\gamma$ greater from the maximal value of $\lambda_{n,k}(I_{j-1})$ for all $j\in\{2,\ldots, k+\frac{n-1}{2}\}$ and we will use this fact heavily.
Furthermore, if we denote $\alpha:=\frac{\gamma}{\scr[n]+\scr[n-1]}$ that is $\alpha$ is the length of the smallest interval of monotonicity, then $\lambda_{n,k}(x_i)=\max \lambda_{n,k}([x_i,x_{i+1}-\alpha])$
for each $i\in \{1,\ldots,k+\frac{n-1}{2}-1\}$, where $x_i$ are as in Observation~\ref{obs:gammaapart}. Thus for the part when $\gamma \leq \diam(A)<2\gamma$ the conclusion follows from Observation~\ref{obs:gammaapart}.\\
Let $A\subset I$ be an interval such that $\diam(A)\geq 2\gamma$, say $A:=[a,b]\subset I$. Then there is maximal $j$ such that $I_j\subset A$.
If $j\geq k+\frac{n-1}{2}$ then $\max \lambda_{n,k}(A)=1$ and so $\lambda_{n,k}([a,b+r])=\lambda_{n,k}([a,b])$ where $b+r\leq 1$ for some non-negative real number $r$.
So assume that $j<k+\frac{n-1}{2}$. Next, let $j'$ be the maximal number such that $x_{j'}\in [a,b+r]${, where $x_{j'}$ is as in Observation~\ref{obs:gammaapart}}. If $j'=k+\frac{n-1}{2}$ then $\lambda_{n,k}(x_j')=1$
and clearly $x_{j'}-x_j<r+\gamma$ so $\lambda_{n,k}([a,b+r])\subset B(\lambda_{n,k}([a,b]),r+\gamma)$.
For the last case, fix any $x\in [b,b+r]$. Then, by the repetative structure of building blocks of the graph $\lambda_{n,k}$ (see Figure~\ref{fig:lambda}), there is a non-negative real number $s$ such that $s\gamma <r+\gamma$,
$y:=x-s\gamma \in [a,b]$ and $\lambda_{n,k}(x)\leq \lambda_{n,k}(y)+s\gamma$. But this shows that $\lambda_{n,k}([a,b+r])\subset B(\lambda_{n,k}([a,b]),r+\gamma)$, which concludes the proof.
\end{proof}

 We denote the set of {\em piecewise linear maps that preserve Lebesgue measure $\lambda$} by $\mathrm{PL}_{\lambda}$ and {\em piecewise linear maps that are leo and preserve $\lambda$} by $\mathrm{PL}_{\lambda\mathrm{(leo)}}$. 
If additionally they satisfy Markov property, which means there is a partition $0=a_0<a_1<\ldots<a_n=1$ such that for each $i$ the map $f_{[a_i,a_{i+1}]}$ is monotone and there are $s<t$ such that
$f([a_i,a_{i+1}])=[a_s,a_t]$, then we call them \emph{Markov piecewise linear leo maps that preserve $\lambda$}. The set of all such maps is denoted $\mathrm{PLM}_{\lambda\mathrm{(leo)}}$.

\begin{defn}
	A piecewise linear map $f\in C(I)$ is called \emph{admissible}, if $|f'(t)|\geq 4$ for every $t\in I$ for which $f'(t)$ exists and $f$ is leo.
\end{defn}

Having the appropriate Lebesgue measure-preserving perturbations $\lambda_{n,k}$ from Lemma~\ref{lem:MincUpdt} we now get the following lemma. 

\begin{lem}\label{lem:LemMT}
Let $f:I\to I$ be an admissible map. Let $\eta$ and $\delta$ be two positive real numbers. Then there is an admissible map $F:I\to I$ and there is a positive integer $n$ such that $F^n$ is $\delta$-crooked and $\rho(F,f)<\eta$. Moreover, if $f\in C_{\lambda}(I)$, such $F$ can be also chosen to be in $C_{\lambda}(I)$.
\end{lem}

\begin{proof}
	We define $F=f\circ \lambda_{n,k}$ and proceed with the proof as in \cite[Lemma]{MT}, since we can replace their map $g$ with $\lambda_{n,k}$ as it follows from Lemma~\ref{lem:MincUpdt}, provided that $n$ and $k$ are properly chosen. As $\lambda_{n,k}$ is a Lebesgue preserving map for all odd $n\ge 7$ and all $k\geq 1$ by Proposition~\ref{prop:LebPres} the moreover part follows by choosing $f\in C_{\lambda}(I)$. 
\end{proof}

The following lemma gives a useful fact about the admissible maps.

\begin{lem}[{\cite[Theorem~10]{KO}}]
	For every $\eps>0$ and every leo map $f \in C(I)$ there exists $F\in C(I)$ such that $F$ is admissible and $\rho(F,f) < \eps$.
\end{lem}

We will need its small adjustment, which can be obtained with the help of the following useful result.
\begin{lem}[{\cite[Proposition~12]{MT}}]\label{lem:BT}
	The set $\mathrm{PLM}_{\lambda\mathrm{(leo)}}$ is dense in $C_{\lambda}(I)$.
\end{lem}

In what follows we will also need perturbations of maps that preserve Lebesgue measure, similarly as in \cite{BT}.

\begin{defn}
For maps $f,g\colon~[a,b]\subset I\to I$ we say that they are \emph{$\lambda$-equivalent} if for each Borel set $A\in \mathcal{B}(I)$ (where $\mathcal{B}(I)$ from now on denotes \emph{Borel $\sigma$-algebra} on $I$) it holds that,
$$\lambda(f^{-1}(A))=\lambda(g^{-1}(A)).
$$
For $f\in C_{\lambda}(I)$ and $[a,b]\subset I$ we denote by $C(f;[a,b])$ the set of all continuous maps $\lambda$-equivalent to $f|_{[a,b]}$. We
define
$$C_*(f;[a,b]):=\{h\in C(f;[a,b])\colon~h(a)=f(a),~h(b)=f(b)\}.$$
\end{defn}

The following definition is illustrated by Figure~\ref{fig:perturbations}.
 
\begin{defn}\label{eq:perturb} Let $f\in C_{\lambda}(I)$ and $[a,b]\subset I$. For any fixed $m\in\N$, let us define the map $h=h\langle f;[a,b],m\rangle\colon~[a,b]\to I$ by ($j\in\{0,\dots,m-1\}$):
\begin{equation*}
h(a + x) := \begin{cases}
f\left (a+m \Big (x-\frac{j(b-a)}{m}\Big) \right )\text{ if } x\in \left [\frac{j(b-a)}{m},\frac{(j+1)(b-a)}{m} \right ],~j\text{ even}, \\
f\left (a+m \Big (\frac{(j+1)(b-a)}{m}-x \Big ) \right )\text{ if } x\in \left [\frac{j(b-a)}{m},\frac{(j+1)(b-a)}{m} \right ],~j\text{ odd}.
\end{cases}
\end{equation*}
Then $h\langle f;[a,b],m\rangle\in C(f;[a,b])$ for each $m$ and $h\langle f;[a,b],m\rangle\in C_*(f;[a,b])$ for each $m$ odd. In particular, if $h=h\langle f;[a,b],m\rangle$, $m$ odd, we will speak of
\emph{regular $m$-fold window perturbation} $h$ of $f$ (on $[a,b]$).
\end{defn}

\begin{figure}[!ht]
	\centering
	\begin{tikzpicture}[scale=4, rotate=180]
	\draw (0,0)--(0,1)--(1,1)--(1,0)--(0,0);
	\draw[thick] (0,1)--(1/2,0)--(1,1);
	\node at (1/2,1/2) {$f$};
	\node at (7/16,-0.1) {$b$};
	\node at (5/8,-0.1) {$a$};
	\draw[dashed] (7/16,0)--(7/16,1/4)--(5/8,1/4)--(5/8,0);
	\node[circle,fill, inner sep=1] at (7/16,1/8){};
	\node[circle,fill, inner sep=1] at (5/8,1/4){};
	\end{tikzpicture}
	\hspace{1cm}
	\begin{tikzpicture}[scale=4, rotate=180]
	\draw (0,0)--(0,1)--(1,1)--(1,0)--(0,0);
	\draw[thick] (0,1)--(7/16,1/8)--(14/32+1/48,0)--(1/2,1/4)--(1/2+1/24,0)--(9/16,1/8)--(9/16+1/48,0)--(5/8,1/4)--(1,1);
	\draw[dashed] (1/2,0)--(1/2,1/4);
	\draw[dashed] (9/16,1/4)--(9/16,0);
	\node at (1/2,1/2) {$h$};
	\node at (7/16,-0.1) {$b$};
	\node at (5/8,-0.1) {$a$};
	\draw[dashed] (7/16,0)--(7/16,1/4)--(5/8,1/4)--(5/8,0);
	\node[circle,fill, inner sep=1] at (1/2,1/4){};
	\node[circle,fill, inner sep=1] at (9/16,1/8){};
	\node[circle,fill, inner sep=1] at (7/16,1/8){};
	\node[circle,fill, inner sep=1] at (5/8,1/4){};
	\end{tikzpicture}
	\caption{For  $f\in C_{\lambda}(I)$ and $a,b\in I$ shown on the left picture, we show on the right picture
 the regular $3$-fold window perturbation of $f$ by $h=h\langle f;[a,b],3\rangle\in C_*(f;[a,b])$.}\label{fig:perturbations}
\end{figure}
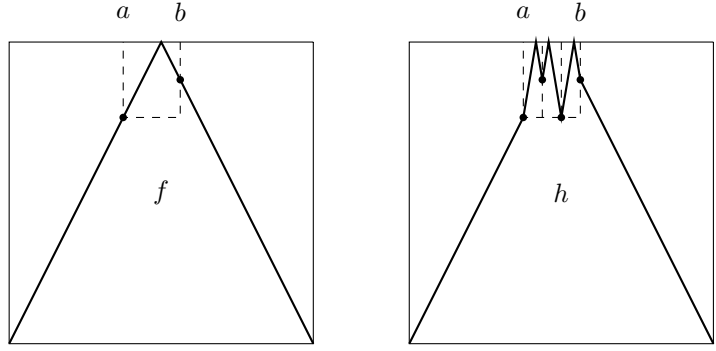

For more details on the perturbations from the previous definition we refer the reader to \cite{BT}.

\begin{lem}\label{lem:LeoApproxByAdm}
	For every $\eps>0$ and every leo map $f \in C_\lambda(I)$ there exists $F\in C_\lambda(I)$ such that $F$ is admissible and $\rho(F,f) < \eps$.
\end{lem}
\begin{proof}
By Lemma~\ref{lem:BT} piecewise linear and Markov leo maps are dense in $C_{\lambda}(I)$, so let us start with such map $g$ with $\rho(g,f) < \eps/2$.
Let us choose a Markov partition $0=a_0<a_1<\ldots<a_n=1$ for $g$. Periodic points are dense for a leo map, so including points from periodic orbits as points in the partition, we may also require that $|a_{i+1}-a_i|<\delta$ for any fixed $\delta>0$.
In particular, we have that $g$ is monotone on each interval $[a_i,a_{i+1}]$, $\diam(g([a_i,a_{i+1}]))<\eps/2$ and $g([a_i,a_{i+1}])=[a_k,a_{k'}]$ for some indices $k<k'$.
Now, repeating construction in Lemma~5 of \cite{BT} we construct a new map $F$ by replacing each $g|_{[a_i,a_{i+1}]}$ by its regular $m$-fold window perturbation, with odd
and sufficiently large $m$. This way $F$ is admissible with $\rho(F,g)\leq \diam(g([a_i,a_{i+1}]))<\eps/2$.
Window perturbations are invariant for $C_\lambda(I)$, hence $F\in C_\lambda(I)$.
Clearly also $g([a_i,a_{i+1}])=F([a_i,a_{i+1}])=[a_k,a_{k'}]$. Therefore, for each $i$ there is $n\in \N$ such that $F^n([a_i,a_{i+1}])=I$.
But then, if we fix any open set $U\subset I$, then since slope on intervals of monotonicity of $F$ is at least $4$, there is $M\in \mathbb{N}$ such that $F^M(U)$ contains three consecutive intervals of monotonicity, and therefore 
$F^{M+1}(U)\supset [a_i,a_{i+1}]$.
\end{proof}

The following lemma is an essential ingredient in the construction of pseudo-arc using inverse limits in \cite{MT}.
\begin{lem}[{\cite[Proposition~2]{MT}}]\label{lem:crooked}
	Let $f,F\in C(I)$ be two maps so that
	$\rho(f,F)<\varepsilon$. If $f$ is $\delta$-crooked, then $F$ is
	$(\delta+2\varepsilon)$-crooked.
\end{lem}

Now we are ready to prove the main theorem of this section.

\begin{proof}[Proof of Theorem \ref{thm:UniLimPresLeb}]
For any $k\geq 1$ let the set $A_k\subset C_\lambda(I)$ be contained in the set of maps $f$ such that $f^n$ is $(1/k-\delta)$-crooked for some $n$ and some sufficiently small $\delta>0$. 
First observe that $A_k$ is dense. Namely,  by Lemma~\ref{lem:BT} it holds that piecewise linear leo Markov maps are dense in $C_{\lambda}(I)$.
If we start with such a map $g$ then first applying Lemma~\ref{lem:LeoApproxByAdm} and next Lemma~\ref{lem:LemMT} we modify $g$ to a map $f\in A_k$
by an arbitrarily small perturbation. But if $f\in A_k$ and $n,\delta$ are constants from the definition of $A_k$, then by Lemma~\ref{lem:crooked} we have $B(f,\delta/4)\subset A_k$.
This shows that $A_k$ contains an open dense set. But then the set
$$
\mathcal{T}:=\bigcap_{k=1}^\infty A_k
$$
is a dense $G_\delta$ and clearly each element $f\in \mathcal{T}$ satisfies the conclusion of the theorem. 
\end{proof}

\subsection{Maps with a dense set of periodic points}

Denote by $C_{DP}(I)\subset C(I)$ the family of \emph{interval maps with a dense set of periodic points}. First note that $C_{DP}(I)$ is not a closed space. However, since $\overline{C_{DP}(I)}$ is closed in $C(I)$ it is thus a complete space as well. Now we state a useful remark that is given and explained in the introduction of \cite{BCOT}.
 
\begin{rem}\label{rem:equivalent}Let $f\in C(I)$. The following conditions are equivalent:
\begin{itemize}
 \item[(i)] $f$ has a dense set of periodic points.
 \item[(ii)] $f$ preserves a nonatomic probability measure $\mu$ with $\supp~\mu=I$. 
    \item[(iii)] There exists a homeomorphism $h$ of $I$ such that $h\circ f\circ h^{-1}\in C_{\lambda}(I)$.
\end{itemize}
\end{rem}

\begin{thm}\label{thm:dense}
There is  a dense $G_\delta$ set $\mathcal{Q}\subset \overline{C_{DP}(I)}$ such that if $g\in \mathcal{Q}$ then
for every $\delta>0$ there exists a positive integer $n$ so that $g^n$ is $\delta$-crooked. 
\end{thm}
\begin{proof}
For a non-atomic probability measure with full support $\mu$ the map $h\colon~I\to I$ defined as $h(x)=\mu([0,x])$ is a homeomorphism of $I$; moreover, if $f$ preserves $\mu$ then $h\circ f\circ h^{-1}\in C_{\lambda}(I)$ (see the proof of Theorem~2 in \cite{BCOT} for more detail on this construction). 
 
 By Remark~\ref{rem:equivalent}, for every $f\in \overline{C_{DP}(I)}$ we have a homeomorphism $h\colon I\to I$ so that $F=h\circ f\circ h^{-1}\in C_{\lambda}(I)$. But for every $\xi>0$ there is $G\in C_{\lambda}(I)$ so that for every $\delta>0$ there exists a positive integer $n$ so that $G^n$ is $\delta$-crooked and $\rho(F,G)<\xi$. But since $h$ is fixed and continuous, for any given $\eps>0$
 there is $\zeta>0$ such that if $\rho(F,G)<\zeta$ then $\rho(h^{-1}\circ F\circ h,h^{-1}\circ G\circ h)<\eps$. Assume for simplicity of notation that $h$ is increasing.
 
 Put $g=h^{-1}\circ G\circ h$. For every $\delta>0$ there is $\gamma$ such that if $|x-y|<\gamma$ then $|h^{-1}(x)-h^{-1}(y)|<\delta$.
 There is $n\in \N$ such that $G^n$ is $\gamma$-crooked. Fix any $c<d$. There are $h(c)<x\leq y<h(d)$ such that $|G^n(h(d))-G^n(x)|<\gamma$ and $|G^n (h(c))-G^n(y)|<\gamma$.
 If we denote $c'=h^{-1}(x)$ and $d'=h^{-1}(y)$ then $|g^n(c)-g^n(d')|<\delta$ and $|g^n(c')-g^n(d)|<\delta$. It means that for any $\delta>0$ there is $n$ such that $g^n$ is $\delta$-crooked.

 We also get $\rho(f,g)<\eps$ and $g\in \overline{C_{DP}(I)}$ since $g$ and $G$ are conjugate maps.
 
 But by Lemma~\ref{lem:crooked}, maps $g\in \overline{C_{DP}(I)}$ such that for every $\delta>0$ there exists a positive integer $n$ so that $g^n$ is $\delta$-crooked form a $G_\delta$ subset.
 Summing up, the set of maps $g\in \overline{C_{DP}(I)}$ such that for every $\delta>0$ there exists a positive integer $n$ so that $g^n$ is $\delta$-crooked is residual in $\overline{C_{DP}(I)}$.
\end{proof}

\section{Lifting one-dimensional dynamics to the invertible dynamics of the plane}\label{sec:extension}
\subsection{Introduction to inverse limits}
Now let us introduce
\emph{inverse limit spaces}, a technique that we will work with from now on. For a collection of continuous maps $f_i:Z_{i+1}\to Z_i$ where $Z_i$ are compact metric spaces for all $i\geq 0$ we define
\begin{equation}
\underleftarrow{\lim} (Z_i,f_i)
:=
\{\hat z:=\big(z_{0},z_1,\ldots \big) \in Z_0\times Z_1,\ldots\big|  
z_i\in Z_i, z_i=f_i(z_{i+1}), \forall i \geq 0\}.
\end{equation}
We equip $\underleftarrow{\lim} (Z_i,f_i)$ with the subspace 
metric induced from the 
\emph{product metric} in $Z_0\times Z_1\times\ldots$,
where $f_i$ are called the {\em bonding maps}.
If $Z_i=Z$ and $f_i=f$ for all $i\geq 0$, the inverse limit space 
$$
\hat Z:=\underleftarrow{\lim} (Z,f)
$$
also comes with a natural homeomorphism, 
called the \emph{natural extension} of $f$ (or the 
\emph{shift homeomorphism})
$\hat f:\hat Z
\to \hat Z$, 
defined as follows. 
For any $\hat z= \big(z_{0},z_1,\ldots \big)\in \hat Z$,
\begin{equation}
\hat{f}(\hat z):= \big(f(z_0),f(z_{1}),f(z_2),\ldots \big) =\big(f(z_0),z_{0},z_1,\ldots \big).
\end{equation}
By $\pi_{i}$ we shall denote
the \emph{$i$-th projection} 
from 
$\hat Z$ to its $i$-th coordinate. 

\subsection{Pseudo-arc and genericity}

In this section we provide consequences of the results obtained in the preceding section. As a tool we need Proposition~4 from \cite{MT} which we state as the following lemma.

\begin{lem}
 Let $f\colon I \to I$ be a continuous map with the property that for every $\eps>0$ there is an integer $n$ such that $f^n$ is $\eps$-crooked. Then $\hat I$ is the pseudo-arc.
\end{lem}

This combined with Theorem~\ref{thm:UniLimPresLeb} proves Corollaries~\ref{cor:MincTransue} and \ref{cor:MincTransue1}.

\begin{rem}
	Later in the paper we will often refer to the dense $G_{\delta}$ set $\mathcal{T}\subset C_{\lambda}(I)$ from Theorem~\ref{thm:UniLimPresLeb}, having in mind that inverse limit with the single bonding map being any map from $\mathcal{T}$ produces the pseudo-arc.
\end{rem}

We will also need the following measure-theoretic definition to state some obvious measure-theoretic consequences of the main theorem of the preceding section.

\begin{defn}\label{def:induced} Let $X$ be a Euclidean space with Lebesgue measure $\lambda$ and let $f\colon X\to X$ be a (surjective) map. An invariant measure $\hat{\mu}_f$ for the natural extension $\hat f\colon \hat X\to \hat X$ is called the \emph{inverse limit physical measure} if $\hat{\mu}_f$ has a basin $\hat B$ so that  $\lambda(\pi_0(\hat B))>0$.
\end{defn}

If we combine Theorem~\ref{thm:UniLimPresLeb} and Corollary~\ref{cor:MincTransue} and results from \cite{BT}, \cite{Li} and \cite{Br} (see also the survey \cite{CL} on dynamical properties that extend to inverse limit spaces) we get Corollary~\ref{cor:Typical}. Note that this corollary also contributes to the study of possible homeomorphisms on the pseudo-arc.

\begin{proof}[Proof of Corollary~\ref{cor:Typical}]
First we intersect $\mathcal{T}$ with the dense $G_{\delta}$ set so that properties from \cite{BT} hold; we obtain a dense $G_{\delta}$ set in $C_{\lambda}(I)$ and we denote it by $\mathcal{T}'$. Recall that $C_{\lambda}(I)$ is a complete space. Thus, by the Alexandrov theorem (\cite{Kur}, p. 408), $\mathcal{T}$ is homeomorphic to a complete space through complete metrization. Even more, the new metric that we define on $\mathcal{T}'$ can be taken so that the topology of $\mathcal{T}'$ with respect to $C_{\lambda}(I)$ is unchanged. Thus, if $f_n\to f$ uniformly in $C_{\lambda}(I)$ for all $\{f_n\}_{n\in \N}, f\in \mathcal{T}\subset C_{\lambda}(I)$, then also $f_n\to f$ uniformly in $\mathcal{T}'$.

Since $f$ is leo it is also transitive and because $f\in C_{\lambda}(I)$ it holds it has a dense set of periodic points. 
The last two items follow directly from \cite{BCOT}.
\end{proof}

The following proposition is (in particular) implied by Theorem~2 from \cite{BCOT} which states that there exist a dense collection of Lebesgue measure-preserving interval maps with Lebesgue measure $1$ on the set of periodic points and positive measure on periodic points of any period $k\geq 1$. The proof of Theorem~2 from \cite{BCOT} constructs a topological conjugacy between a dense collection of generic Lebesgue preserving maps in $C_{\lambda}(I)$ (which we have shown that have iterates being $\delta$-crooked for any $\delta>0$) and the maps with the former property stated in this paragraph.

\begin{prop} 
There exists a transitive homeomorphism on the pseudo-arc $P$ which preserves induced physical inverse limit measure $\hat{m}$ on $P$ with measure $1$ on the set of periodic points and positive measure on periodic points of any period $k\geq 1$.
\end{prop}

\section{A one-parameter family of pseudo-arc attractors with continuously varying prime end rotation numbers.}\label{sec:arc}

 In this section we will construct a parametrized family of pseudo-arc attractors that vary continuously with one parameter. We will start with a particular piecewise linear family that varies continuously and has appropriate properties for a subsequent treatment; then we will repeatedly perturb the whole family with the same perturbation to obtain in the uniform limit a sufficiently crooked family of maps. Then we will apply the BBM procedure to obtain a continuously varying parametrized family of sphere homeomorphisms with the pseudo-arc attractors. Let us note that there are many non-conjugate families of interval maps that satisfy properties below and we could have picked them as a starting point. On the other hand, there is a priori no guarantee that a differently chosen family of interval maps will give us a different family of pseudo-arc attractors due to the subsequent application of particular perturbations.
 
 For what follows we refer the reader to Figure~\ref{fig:parametrized}. For any $t\in[0,1]$ let $\tilde{f}_{t}$ be defined by
 $\tilde{f}_{t}(\frac{2}{7})=\tilde{f}_{t}(\frac{4}{7})=\tilde{f}_{t}(\frac{17}{21})=\tilde{f}_{t}(1)=0$ and $\tilde{f}_{t}(\frac{3}{7})=\tilde{f}_{t}(\frac{5}{7})=\tilde{f}_{t}(\frac{19}{21})=1$ and piecewise linear between these points on the interval $[\frac{2}{7},1]$. Furthermore on the interval $x\in [0,\frac{2}{7}]$ let:

\begin{equation}\label{eq:param}
\tilde{f}_{t}(x)=
\begin{cases}
7(x-t\frac{4}{21});& x\in (1-t)[0,\frac{1}{7}]+t\frac{4}{21},\\
1-7(x-\frac{1}{7}(1-t)-t\frac{4}{21});& x\in (1-t)[\frac{1}{7},\frac{2}{7}]+t\frac{4}{21},\\
\frac{21}{2}(x-t\frac{2}{21});& x\in t[\frac{2}{21},\frac{4}{21}],\\
1-\frac{21}{2}x;& x\in t[0,\frac{2}{21}],\\
-\frac{21}{2}(x-\frac{2}{7});& x\in [\frac{2}{7}-t\frac{2}{21},\frac{2}{7}],
\end{cases}
\end{equation}

see Figure~\ref{fig:parametrized} to see graphs of three special parameters in this family.

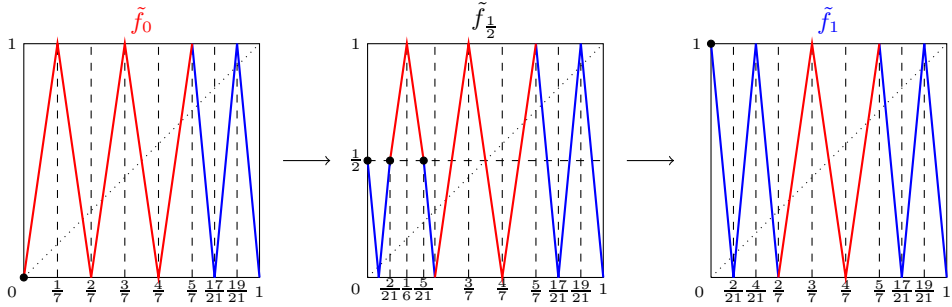
\begin{figure}[!ht]
	\begin{tikzpicture}[scale=3.1]
	\draw (0,0)--(0,1)--(1,1)--(1,0)--(0,0);
	\draw[dotted] (0,0)--(1,1);
	\draw[dashed] (1/7,0)--(1/7,1);
	\draw[dashed] (2/7,0)--(2/7,1);
	\draw[dashed] (3/7,0)--(3/7,1);
	\draw[dashed] (4/7,0)--(4/7,1);
	\draw[dashed] (5/7,0)--(5/7,1);
	\draw[dashed] (17/21,0)--(17/21,1);
	\draw[dashed] (19/21,0)--(19/21,1);
	\node at (-0.05,-0.06) {\tiny $0$};
	\node at (1/7,-0.06) {\tiny $\frac{1}{7}$};
	\node at (2/7,-0.06) {\tiny $\frac{2}{7}$};
	\node at (3/7,-0.06) {\tiny $\frac{3}{7}$};
	\node at (4/7,-0.06) {\tiny $\frac{4}{7}$};
	\node at (5/7,-0.06) {\tiny $\frac{5}{7}$};
	\node at (17/21,-0.06) {\tiny $\frac{17}{21}$};
	\node at (19/21,-0.06) {\tiny $\frac{19}{21}$};
	\node at (1,-0.06) {\tiny$1$};
	\node at (-0.05,1) {\tiny $1$};
	\draw[thick,red] (0,0)--(1/7,1)--(2/7,0);
	\draw[thick,blue] (1,0)--(19/21,1)--(17/21,0)--(5/7,1);
	\draw[thick,red] (2/7,0)--(3/7,1)--(4/7,0)--(5/7,1);
	\node at (1/2,1.1) {\small {\color{red}$\tilde{f}_{0}$}};
	\draw[->] (1.1,1/2)--(1.3,1/2);
	\draw[fill] (0,0) circle (0.015);
	\end{tikzpicture}
	\begin{tikzpicture}[scale=3.1]
	\draw (0,0)--(0,1)--(1,1)--(1,0)--(0,0);
	\draw[dotted] (0,0)--(1,1);
	\draw[dashed] (1/6,0)--(1/6,1);
	\draw[dashed] (2/21,0)--(2/21,1/2);
	\draw[dashed] (5/21,0)--(5/21,1/2);
	\draw[dashed] (2/7,0)--(2/7,1);
	\draw[dashed] (3/7,0)--(3/7,1);
	\draw[dashed] (4/7,0)--(4/7,1);
	\draw[dashed] (5/7,0)--(5/7,1);
	\draw[dashed] (17/21,0)--(17/21,1);
	\draw[dashed] (19/21,0)--(19/21,1);
	\draw[dashed] (0,1/2)--(1,1/2);
	\node at (-0.05,-0.05) {\tiny $0$};
	\node at (2/21,-0.05) {\tiny $\frac{2}{21}$};
	\node at (5/21,-0.05) {\tiny $\frac{5}{21}$};
	\node at (1/6,-0.05) {\tiny $\frac{1}{6}$};
	\node at (3/7,-0.05) {\tiny $\frac{3}{7}$};
	\node at (4/7,-0.05) {\tiny $\frac{4}{7}$};
	\node at (5/7,-0.06) {\tiny $\frac{5}{7}$};
    \node at (17/21,-0.06) {\tiny $\frac{17}{21}$};
	\node at (19/21,-0.06) {\tiny $\frac{19}{21}$};
	\node at (1,-0.05) {\tiny$1$};
	\node at (-0.05,1/2) {\tiny$\frac{1}{2}$};
	\node at (-0.05,1) {\tiny $1$};
	\draw[thick,blue] (0,1/2)--(1/21,0)--(2/21,1/2);
	\draw[thick,blue] (5/21,1/2)--(2/7,0);
	\draw[thick,red] (2/21,1/2)--(1/6,1)--(5/21,1/2);
	\draw[thick,blue] (1,0)--(19/21,1)--(17/21,0)--(5/7,1);
	\draw[thick,red] (2/7,0)--(3/7,1)--(4/7,0)--(5/7,1);
	\node at (1/2,1.1) {\small $\tilde{f}_{\frac{1}{2}}$};
	\draw[->] (1.1,1/2)--(1.3,1/2);
	\draw[fill] (0,1/2) circle (0.015);
	\draw[fill] (2/21,1/2) circle (0.015);
	\draw[fill] (5/21,1/2) circle (0.015);
	\end{tikzpicture}
	\begin{tikzpicture}[scale=3.1]
	\draw (0,0)--(0,1)--(1,1)--(1,0)--(0,0);
	\draw[dotted] (0,0)--(1,1);
    \draw[dashed] (2/21,0)--(2/21,1);
	\draw[dashed] (4/21,0)--(4/21,1);
	\draw[dashed] (2/7,0)--(2/7,1);
	\draw[dashed] (3/7,0)--(3/7,1);
	\draw[dashed] (4/7,0)--(4/7,1);
	\draw[dashed] (5/7,0)--(5/7,1);
	\draw[dashed] (17/21,0)--(17/21,1);
	\draw[dashed] (19/21,0)--(19/21,1);
	\node at (-0.05,-0.06) {\tiny $0$};
	\node at (2/21,-0.06) {\tiny $\frac{2}{21}$};
	\node at (4/21,-0.06) {\tiny $\frac{4}{21}$};
	\node at (2/7,-0.06) {\tiny $\frac{2}{7}$};
	\node at (3/7,-0.06) {\tiny $\frac{3}{7}$};
	\node at (4/7,-0.06) {\tiny $\frac{4}{7}$};
	\node at (5/7,-0.06) {\tiny $\frac{5}{7}$};
    \node at (17/21,-0.06) {\tiny $\frac{17}{21}$};
	\node at (19/21,-0.06) {\tiny $\frac{19}{21}$};
	\node at (1,-0.06) {\tiny$1$};
	\node at (-0.06,1) {\tiny $1$};
	\draw[thick,blue] (0,1)--(2/21,0)--(4/21,1)--(2/7,0);
	\draw[thick,blue] (1,0)--(19/21,1)--(17/21,0)--(5/7,1);
	\draw[thick,red] (2/7,0)--(3/7,1)--(4/7,0)--(5/7,1);
	\node at (1/2,1.1) {\small {\color{blue}$\tilde{f}_{1}$}};
	\draw[fill] (0,1) circle (0.015);
	\end{tikzpicture}
	\caption{Graphs of maps $\tilde{f}_0$, $\tilde{f}_{\frac{1}{2}}$ and $\tilde{f}_1$.}\label{fig:parametrized}
\end{figure}

\begin{prop}
For every $t\in[0,1]$, the map $\tilde{f}_{t}\in C_{\lambda}(I)$.
\end{prop}

\begin{proof}
Applying Observation~\ref{l:3} it clearly holds that $\tilde{f}_0,\tilde{f}_1\in C_{\lambda}(I)$. 
For $x\in [0,\frac{2}{7}]$ it holds that $s_0:=|\tilde{f}'_0(x)|=7$ and $s_1:=|\tilde{f}'_1(x)|=\frac{21}{2}$; thus $s_1/s_0=3/2$. 
Note that for any $t\in (0,1)$ it holds that for $x\in [0,\frac{2}{7}]$ and $y\in I$ either there exist $3$ points of $\tilde{f}_{t}^{-1}$ in $[0,\frac{2}{7}]$ where $\tilde f_t$ has slope $\frac{21}{2}$ or $2$ points where $\tilde{f_t}$ has slope $7$.
Therefore, invoking Observation~\ref{l:3} it also follows that $\tilde{f}_{t}\in C_{\lambda}(I)$ for all $t\in (0,1)$.
\end{proof}

\begin{obs}\label{obs:periodicpoints}
$\{\tilde{f}_{t}\}_{t\in[0,1]}$ is a family of continuous piecewise linear maps varying continuously with $t\in[0,1]$.
Furthermore, $\tilde{f}_{0}$ is an $8$-fold map with $\tilde{f}_{0}(0)=0$ and $\tilde{f}_{1}$ is a $9$-fold map so that $\tilde{f}_1(0)=1$ and $\tilde{f}_1(1)=0$. Moreover, since it holds that $\lambda_{n,k}(0)=0$ and $\lambda_{n,k}(1)=1$ for any odd $n$ and $k\in \N$ it holds for $\tilde{g}_{t}:=\tilde{f}_{t}\circ \lambda_{n_1,k_1}\circ\ldots \circ \lambda_{n_m,k_m}$ for any $k_1,\ldots, k_m\geq 1$ and odd $n_1,\ldots, n_m\geq 1$ that $\tilde{g}_0(0)=0$, $\tilde{g}_1(0)=1$ and $\tilde{g}_1(1)=0$.
\end{obs}

\begin{obs}\label{obs:slope}
For every $t\in[0,1]$ and for all points $x\in I$ where $\tilde{f}'_{t}$ is defined it holds that $\frac{21}{2}\geq |\tilde{f}_{t}'(x)|\geq 7$.
\end{obs}

Due to the previous observation we obtain the following.

\begin{obs}\label{obs:diamA}

For every $t\in[0,1]$ and any subinterval $A\subset I$ which does not contain two critical points it holds that $\diam(\tilde{f}_{t}(A))>3\diam(A)$.
\end{obs}

\begin{prop}\label{prop:leo}
For every $t\in[0,1]$, the map $\tilde{f}_{t}$ is leo.
\end{prop}
\begin{proof}
Fix any nondegenerate interval $[a,b]\subset I$ and any $t\in [0,1]$. By Observation \ref{obs:diamA}, there is $N\in \N$ such that $\tilde{f}^N_{t}([a,b])$ contains two critical points. But then, by the definition of $\tilde{f}_t$,
we obtain that $\tilde{f}^{N+1}_{t}([a,b])=I$.
Indeed $\tilde f_t$ is leo for all $t\in[0,1]$.
\end{proof}

Combining Observation~\ref{obs:slope} and Proposition~\ref{prop:leo} we obtain the following corollary.

\begin{cor}
For every $t\in[0,1]$, the map $\tilde{f}_{t}$ is admissible.
\end{cor}

Now we will perturb the maps $\{\tilde{f}_{t}\}_{t\in[0,1]}$ with the help of Lemma~\ref{lem:MincUpdt} using the same perturbations for the whole family. In this way we will get a family 
$\{f_{t}\}_{t\in[0,1]}\subset C_{\lambda}(I)$ of continuous maps varying continuously with $t$ as well. 

\begin{lem}\label{lem:betag}
Let $\beta>0$. Let $\tilde{g}_{t}:=\tilde{f}_{t}\circ \lambda_{n_1,k_1}\circ\ldots \circ \lambda_{n_m,k_m}$ for some $k_1,\ldots, k_m\geq 1$ and odd $n_1,\ldots, n_m\geq 1$. There is an integer $N\geq 0$ so that for every $0\leq a\leq b\leq 1$ with $b-a>\beta$ and all $t\in[0,1]$ it holds that $\tilde{g}^N_{t}([a,b])=I$.
\end{lem}
\begin{proof}
    Take any $N$ so that $3^N\beta>1$. By Observation~\ref{obs:diamA} and since $\lambda_{n,k}$ does not shrink intervals (see Lemma~\ref{lem:LemMT} \eqref{Lcr:3}), there is $j<N$ so that $\tilde{g}_{t}^j([a,b])$ contains two critical points of $\tilde f_{t}$. By the definition of $\tilde f_{t}$ we have $\tilde{g}_{t}^{j+1}([a,b])=I$.
\end{proof}

\begin{lem}\label{lem:LemMTadjusted}
Let $\eta$ and $\delta$ be two positive real numbers fixed for the whole family $\{\tilde{g}_{t}\}_{t\in[0,1]}$ where $\tilde{g}_{t}:=\tilde{f}_{t}\circ \lambda_{n_1,k_1}\circ\ldots \circ \lambda_{n_m,k_m}$ for some $k_1,\ldots, k_m\geq 1$ and odd $n_1,\ldots, n_m\geq 1$. Then there is a positive integer $N$ such that for every $t\in [0,1]$ there exists an admissible map $\tilde{G}_{t}:I\to I$ so that $\tilde{G}_{t}^N$ is $\delta$-crooked and $\rho(\tilde{G}_{t},\tilde{f}_{t})<\eta$. Moreover, $\tilde{G}_t\in C_{\lambda}(I)$ and $\tilde{G}_t=\tilde{g}_t\circ \lambda_{n_{m+1},k_{m+1}}$ for some $k_{m+1}\geq 1$ and odd $n_{m+1}\geq 1$.
\end{lem}
{
\begin{proof}[Sketch of proof.]
The proof is a direct adaptation of the proof of Lemma in \cite{MT}. Let us explain the preparatory part of the proof. The role of $f$ in the Lemma from \cite{MT}
is played by the maps $\tilde{g}_{t}$. Observe that $\lambda_{n_1,k_1}\circ\ldots \circ \lambda_{n_m,k_m}$ remain unchanged in the formula for $\tilde{g}_{t}$ for each $t$, so let us fix $\alpha'>0$ which is the length of the shortest interval of monotonicity of $\lambda_{n_1,k_1}\circ\ldots \circ\lambda_{n_m,k_m}$. Let us fix $\zeta>0$ which is the upper bound of the slope of the map $\lambda_{n_1,k_1}\circ\ldots \circ \lambda_{n_m,k_m}$ and let $\zeta'>0$
be such that $\tilde{f}_{t}$ has at most one critical point in any interval of length at most $\zeta'$. We can choose uniform $\zeta'$ for each $\tilde{f}_{t}$ because all intervals of monotonicity except the left-most are ``uniformly large''.
If we take $\alpha<\min\{\alpha'/\zeta,\zeta'\}$ and any interval $J\subset I$ with $\diam(J)<\alpha$ then $\diam (\tilde{f}_{t}(J))<\alpha'$ so the interval $\tilde{f}_{t}(J)$ contains at most one critical point of the map $\lambda_{n_1,k_1}\circ\ldots \circ\lambda_{n_m,k_m}$. If $J$ does not contain a critical point of $\tilde{f}_{t}$ then $\tilde{g}_{t}$ has at most one critical point in $J$.
But if $J$ contains a critical point of $\tilde{f}_{t}$, say $c\in J$, then it follows from the definition of the maps $\tilde{f}_{t}$ that $\tilde{f}_{t}(c)\in \{0,1\}$. Then $\tilde{f}_{t}(J)\subset [0,\alpha')\cup (1-\alpha',1]$ and in this set 
$\lambda_{n_1,k_1}\circ\ldots \circ \lambda_{n_m,k_m}$ does not have a critical point, so again $\tilde{g}_{t}$ has a unique critical point in $J$.

By the above explanation, similarly as in Lemma from \cite{MT}, if we take any $b - a < \alpha$, then between $a$ and $b$ there is a point $c=c(t)$ such that each $\tilde{g}_{t}$ is linear on both intervals $[a, c]$ and $[c, b]$. All the maps $\tilde{g}_{t}$ have slopes bounded from the above by the same constant, call it $s$, since slopes of all $\tilde{f}_{t}$ are uniformly bounded from the above and all the maps in the composition are piecewise linear. Also, Lemma~\ref{lem:betag}
provides the same $N$ for all $\tilde{g}_{t}$ which plays the role of Proposition~6 in \cite{MT}. This defines required bounding constants $\eps<\eta/s$ and $\gamma< \min\{\alpha, s^{-n} , \eps/4, \delta s^{-n}  /5\}$ from the proof of Lemma in \cite{MT}.

The role of $g$ in the proof of the Lemma is played by $\lambda_{n_{m+1},k_{m+1}}$, where sufficiently large values of $n_{m+1},k_{m+1}$ are deduced from Lemma~\ref{lem:MincUpdt} similarly to the application of Proposition~5 for the choice of $g$ in \cite{MT} (using the corresponding $\gamma$ and $\eps$).

After these preparations, the rest of the proof is performed by following exactly the same argument as in Lemma of \cite{MT}, with the only difference that instead of Proposition 5 there, we apply analogous properties of $\lambda_{n_{m+1},k_{m+1}}$ provided by Lemma~\ref{lem:MincUpdt}.
\end{proof}}

\begin{thm}
There exists a family $\{f_{t}\}_{t\in[0,1]}\subset \mathcal{T}$ of maps continuously varying with $t$.
\end{thm}

\begin{proof}
 The procedure we take is the same as in the proof of Theorem~\ref{thm:UniLimPresLeb} however we apply the same perturbations for the whole family $\{\tilde{f}_{t}\}_{t\in[0,1]}$ on every step. To get crookedness and leo on every step we will need to repeatedly use Lemma~\ref{lem:LemMTadjusted}.
 Recall that for any $k\geq 0$ the set $A_k\subset C_\lambda(I)$ is contained in the set of maps $f$ such that $\tilde{f}^{m}$ is $(1/k-\delta)$-crooked for some $m$ and some sufficiently small $\delta>0$.
Starting with $\{\tilde{f}_{t}\}_{t\in[0,1]}$ we use Lemma~\ref{lem:LemMTadjusted} directly to obtain maps $\{\tilde{F}_{t}\}_{t\in[0,1]}\in A_1$.
But if $\{\tilde{F}_{t}\}_{t\in[0,1]}\in A_1$ and $m,\delta$ are constants from the definition of $A_1$, then by Lemma~\ref{lem:crooked} we have $B(\{\tilde{F}_{t}\}_{t\in[0,1]},\delta/4)\subset A_1$. For the second step we take the family $\{\tilde{F}_{t}\}_{t\in[0,1]}$. 
Proceeding as in the rest of the proof of Theorem~\ref{thm:UniLimPresLeb}, ensuring sufficiently fast convergence, we obtain in the intersection of sets $A_k$ the family $\{f_{t}\}_{t\in[0,1]}\subset \mathcal{T}$ of continuous maps varying with $t$.
\end{proof}

Now we will briefly describe standard parametrized BBM construction for the family $\{f_{t}\}_{t\in[0,1]}$, see \cite{BdCH} for more detail.
Let $D\subset \mathbb{R}^2$ be a topological disk, $I\subset D$ is a {\em boundary retract}; i.e. there is a continuous map 
$\alpha:\partial D\times [0,1]\to D$ 
which decomposes $D$ into a continuously varying family of arcs 
$\{\alpha(x,\cdot)\}_{x\in \partial D}\subset \mathcal C(I, D)$, so that $\alpha(x,\cdot)(I)$ are pairwise disjoint except perhaps at the endpoints $\alpha(x,1)$, where one requires that $\alpha(x,1)\in I$. 
We can then associate a {\em retraction} $r\colon D \to I$ defined by $r(\alpha(x,s))= \alpha(x,1)$ for every $x\in \partial D$ corresponding to the given decomposition. The map is boundary retract, but we need
to maintain the disc, so we will collapse only the ``inner half'' of it (see definition of $R$ below).
Recall also, that a continuous map between two compact metric spaces is called a \emph{near-homeomorphism}, if it is a uniform limit of homeomorphisms.

Having the above decomposition in arcs we define \emph{smash} $R:D\to D$ as a near-homeomorphism so that:
$$
R(\alpha(x,s))=
\begin{cases}
\alpha(x,2s);& s\in [0,1/2],\\
\alpha(x,1); & s\in [1/2,1].
\end{cases}
$$

We define the \emph{unwrapping} of $\{f_{t}\}_{t\in[0,1]}\subset \mathcal{T}$ as a continuously varying family $\bar f_{t}\colon D\to D$ of orientation-preserving homeomorphisms so that for all $t$:
\begin{enumerate}[(i)]
\item\label{unwrapping:i} $\supp \bar{f}_{t}\subset \{\alpha(x,s); s\geq 1/2\}$,
\item\label{unwrapping:ii} $R\circ \bar{f}_{t}|_{I}= f_{t}$,
\end{enumerate}
For the purpose of easier discussion afterwards let us fix the unwrappings to be the ``rotated graphs'' of the corresponding functions, following \cite{BdCH}, and we will call such unwrappings {\em standard}.
 We additionally require that $\bar{f}_0$ preserves a horizontal radial arc that connects $0\in I$ to $\partial D$ (which we can require by Observation~\ref{obs:periodicpoints}). Also, we require (again using Observation~\ref{obs:periodicpoints}) that $\bar{f}_1$ interchanges horizontal radial arcs that connect $0\in I$ to $\partial D$ and $1\in I$ to $\partial D$.\\
As a consequence of \eqref{unwrapping:i} we obtain that for all $x\in \partial D$ and $s\in [0,1/2]$ we have $\bar{f}_{t}(\alpha(x,s))=\alpha(x,s)$.
 At this point we would like to stress we do not claim that all unwrappings associated to some map $f\in \mathcal{T}$ are  dynamically equivalent (see Definition~\ref{def:equivalent}).
Now set $H_{t}=R\circ \bar{f}_{t}$ which is a near-homeomorphism. By Brown's theorem \cite{Br}, $\hat D_{t}:=\underleftarrow{\lim}(D,H_{t})$ is a closed topological disk; i.e. there exists a homeomorphism $h_{t}\colon \hat D_{t}\to D$. Let $\Phi_t:=h_{t}\circ \hat H_{t}\circ h_{t}^{-1}\colon D\to D$ and let $\Lambda_t:=h_{t}(\hat I_t)$. It follows from \cite{BM} that $\Phi_t|_{\Lambda_t}$ is topologically conjugate to $\hat{f}_{t}: \underleftarrow{\lim} (I, f_{t}) \to \underleftarrow{\lim} (I, f_{t})$ for every $t\in [0,1]$.  Moreover, it follows from the choice of unwrapping that every point from the interior of $\hat D_t$ is attracted to $\hat I_{t}$, therefore, $\hat I_{t}$ is a global attractor for $\hat H_{t}$ and thus $\Lambda_t$ is a global attractor for $\Phi_t$ as well. By Theorem 3.1 from \cite{3G-BM} $\{\Phi_t\}_{t\in[0,1]}$ vary continuously with $t$ and the attractors $\{\Lambda_t\}_{t\in [0,1]}$ vary continuously in Hausdorff metric.

To a non-degenerate and non-separating continuum $K\subset D\setminus \partial D$ we can associate the {\em circle of prime ends} $\mathbb{P}$ as the compactification of $D\setminus K$.
If $h\colon\mathbb{R}^2\to \mathbb{R}^2$ preserves orientation and $h(K)=K$, $h(D)=D$ then $h$ induces an orientation preserving homeomorphism $\tilde{H}\colon\mathbb{P}\to\mathbb{P}$, and therefore it gives a natural \emph{prime ends rotation number}. 
In what follows we will also need the following result by Barge~\cite{Barge}. 

\begin{lem}[Proposition 2.2 in \cite{Barge}]\label{Barge}
Suppose that $\{\Psi_t\}_{t\in[0,1]}$ is a family of orientation-preserving  homeomorphisms on a topological disk $D\subset \mathbb{R}^2$ continuously varying with $t$. 
For every $t\in[0,1]$ let $K_t \subset \Int{D}$ be a 
non-degenerate sphere non-separating continuum, invariant under $\Psi_t$,
and assume that $\{K_t\}_{t\in[0,1]}$ vary continuously with $t$ in the Hausdorff metric. 
Then the prime ends rotation numbers vary continuously with $t\in [0,1]$.   	
\end{lem}

Finally, let us define how we distinguish the embeddings from the dynamical perspective. 
In what follows we generalize the definition from \cite{BdCH} of equivalence of embeddings.

\begin{defn}\label{def:equivalent}
Let $X$ and $Y$ be metric spaces. Suppose that $F:X\to X$ and $G:Y\to Y$ are homeomorphisms and $E:X\to Y$ is an embedding. If $E\circ F=G\circ E$ we say that the embedding $E$ is a {\em dynamical embedding} of $(X,F)$ into $(Y,G)$. If $E$, resp. $E'$, are dynamical embeddings of $(X,F)$ resp. $(X',F')$ into $(Y,G)$, resp. $(Y',G')$, and there is a homeomorphism $H:Y\to Y'$ so that $H(E(X))=E'(X')$ which conjugates $G|_{E(X)}$ with $G'|_{E'(X')}$ we say that the embeddings $E$ and $E'$ are {\em dynamically equivalent}.
\end{defn}
\begin{rem}
In our case $Y=Y'=\mathbb{R}^2$ and $X,X'$ are pseudo-arcs (in particular plane non-separating continua). Thus, the dynamical equivalence from Definition~\ref{def:equivalent} induces a conjugacy on the circles of prime ends without requiring that $H$ conjugates $G$ with $G'$ on all $\mathbb{R}^2$.
\end{rem}

We will also use the following definition.

\begin{defn}
We say that a point $x\in K\subset \mathbb{R}^2$ is {\em accessible} if there exists an arc $A\subset \mathbb{R}^2$ such that $A\cap K=\{x\}$ and $A\setminus \{x\}\subset \mathbb{R}^2\setminus K$.
\end{defn}

Now let us prove the main theorem of this section.

\begin{proof}[Proof of Theorem \ref{lem:BBM1}]
Items (a) and (b) follow directly from Theorem 3.1 of \cite{3G-BM}.

Let us argue that $\Lambda_0=h(\hat I_0)$ has an accessible point $h_{0}((0,0,\ldots))$ fixed under $\Phi_0$. 
We choose a horizontal radial arc $Q_0\subset D$ that has an endpoint in $0\in I\subset D$. Note that by the definition of $H_0$ it holds that $H_0(Q_0)=Q_0$ and $H_0|_{Q_0}$ is a near-homeomorphism. 
Thus, $J_0:=\underleftarrow{\lim}(Q_0,H_0|_{Q_0})$ is an arc by the result of Brown \cite{Bro}, as it is an inverse limit of arcs with near-homeomorphisms for bonding maps. Therefore, $\Phi_0(h_0(J_0))=h_0(J_0)$ and thus $\Lambda_0$ has an accessible fixed point which is connected to $\partial D$ by an invariant arc and thus it defines a prime end $P_0\in \mathbb{P}_0$ on the corresponding circle of prime ends $\mathbb{P}_0$. Since $\Lambda_0$ is the pseudo-arc and thus an indecomposable plane non-separating continuum, Theorem 5.1 from \cite{Bre} implies that $P_0$ is a fixed point of the induced homeomorphism $\tilde{H}_0:\mathbb{P}_0\to\mathbb{P}_0$. Therefore, the prime ends rotation number of $\tilde{H}_0$ is $0$.\\
Now let us show that the rotation number of the induced prime end homeomorphism $\tilde{H}_1:\mathbb{P}_1\to\mathbb{P}_1$ corresponding to $H_1$ is $1/2$. Similarly as above, we see that there are two accessible points $p_1,p'_1\in\Lambda_1$ such that $H_1(p_1)=p'_1$ and $H^2_1(p_1)=p_1$. Therefore, there are corresponding prime ends $P_1,P'_1\in \mathbb{P}_1$. By Theorem 3.2. from \cite{Bre}, if a point from an indecomposable continuum is accessible it corresponds to a unique prime end, thus $P_1,P'_1$ are the only prime ends corresponding to accessible points $p_1$ and $p'_1$ respectively. Furthermore, Theorem 5.1 from \cite{Bre} implies that $\tilde{H}^2_1(P_1)=P_1$ and $\tilde{H}^2_1(P'_1)=P'_1$. We only need to exclude that $\tilde{H}_1(P_1)=P_1$ ($\tilde{H}_1(P'_1)=P'_1$). But if $\tilde{H}_1(P_1)=P_1$ ($\tilde{H}_1(P'_1)=P'_1$), the definition of the map $\tilde{H}_1$ would imply that $p_1$ ($p'_1$) and $H_1(p_1)$ ($H_1(p'_1)$) have the same associated equivalence classes of sequences of crosscuts which leads to a contradiction.
This means that the prime ends rotation number associated to the homeomorphism $\Phi_1$ is $1/2$. Applying Lemma~\ref{Barge} we obtain item (c). 

To show item (d) it is enough to use item (c) and observe that if $\Lambda_t$ and $\Lambda_{t'}$ for $t\neq t'$ are embedded dynamically equivalently, then also the prime end homeomorphisms $\tilde H_t$ and  $\tilde H_{t'}$ associated to $\Lambda_t$ and $\Lambda_{t'}$ are conjugated (because the associated equivalence classes of sequences of crosscuts are interchanged by the conjugating homeomorphism) which implies the equality of the associated prime ends rotation numbers. 
\end{proof}

\begin{rem}
While the embeddings from Theorem~\ref{lem:BBM1} are different dynamically we can not easily claim that they are different also from the topological point of view. On the other hand, result (d) from Theorem~\ref{lem:BBM1} implies that the parameter space $[0,1]$ is indeed not degenerate. It would be interesting to know how boundary dynamics of the family $\{\Lambda_t\}_{t\in [0,1]}$ looks like precisely (i.e. to understand the sets of accessible points and the prime ends structure), however we do not delve in that aspect of research in this work. 
\end{rem}

\section{Measure-theoretic BBM embeddings}\label{sec:BBM2}

 Note that the set $C_{\lambda}(I)$ is a complete space in $C(I)$ with the supremum metric. However, the space $C_{\lambda}(I)$ is not equicontinuous and thus by Arzel\'a-Ascoli theorem $C_{\lambda}(I)$ is not compact. Therefore, we cannot apply the parametrized BBM construction from \cite{3G-BM} directly to get a parametrized family of planar homeomorphisms varying continuously with the parameter (we could apply construction from \cite{3G-BM} for some compact subset of $C_{\lambda}(I)$ but a priori only from the topological perspective). Thus this section can be viewed as a generalization of the preceding section with the additional measure-theoretic ingredients. 
 
 \subsection{Measure-theoretic preliminaries}\label{subsec:MTpreliminaries}
In this subsection we give some measure-theoretic results that are required later in the construction.
Suppose $X$ is a compact metric space and that $f\colon X\to X$ is continuous and onto and recall that we denote by $\hat X:=\underleftarrow{\lim}(X,f)$ and by $\pi_n:\hat X\to I$ the coordinate projections maps. Recall also that $\mathcal{B}(X)$ denotes the $\sigma$-algebra of Borel sets in $X$. 
First we will need the following standard result. 

\begin{thm}[Theorem 3.2, p. 139 from \cite{Parth}]\label{thm:partha}
Suppose $(X, \mathcal B(X))$ is a separable Borel space and that $f\colon X \to X$ is onto and $\mathcal B(X)$-measurable. Let $\mathcal B(\hat X)$ be the smallest $\sigma$-algebra on $\hat X$ such that all the projection maps $\pi_i$ are measurable. If $\{\mu_n\}_{n\in \N_0}$ is a sequence of probability measures on $\mathcal B(X)$ such that $\mu_n(A) = \mu_{n+1}(f^{-1}(A))$ 
for all $A\in \mathcal B(X)$, then there exists a unique probability measure $\hat\mu$ on $\mathcal B(\hat X)$ 
such that $\hat\mu(\pi^{-1}_n(A))=\mu_n(A)$
for all $A\in \mathcal B(X)$ and each $n\in \N_0$.
\end{thm}

Another result that we use is from \cite{KRS}. Let $M(X)$ denote the \emph{set of all invariant probability measures on the Borel $\sigma$-algebra $\mathcal{B}(X)$}. For any $\mu\in M(X)$ a continuous function $f\colon X\to X$ induces a map $f_{*}\colon M(X)\to M(X)$ given by
$$
f_{*}\mu:=\mu\circ f^{-1}.
$$
By Theorem~\ref{thm:partha} each $(\mu_0,\mu_1,\ldots)\in \underleftarrow{\lim}(M(X),f_{*})$ can be uniquely extended to a probability measure on $\hat X$, that is we have a function:
$$
\mathcal{G}\colon \underleftarrow{\lim}(M(X),f_{*})\to M(\hat X).
$$

Theorem~6 in \cite{KRS} shows that $\mathcal{G}$ is one-to-one and onto. Furthermore, we have the following result that we will use often.

\begin{thm}[Theorem 7 from \cite{KRS}]\label{thm:5.2}
    Suppose $f\colon X\to X$ is a continuous function on a compact metric space. Let $\mathcal{B}(\hat X)$ be the smallest $\sigma$-algebra such that all the projection maps $\pi_i$ are measurable and let $\hat\mu=\mathcal{G}((\mu_0,\mu_0,\ldots))$. Then $\hat\mu$ is $\hat f$-invariant and $\sigma$-invariant.
\end{thm}

\begin{defn} 
Let $\mu$ be an $f$-invariant measure on $X$. Set $B_{\mu}$ is a \emph{basin of $\mu$ for $f$} if for all $g\in C(X)$ and $x\in B_{\mu}$:
$$
\lim_{n\to \infty} \frac{1}{n}\sum_{i=1}^{n}g(f^{i-1}(x))=\int g \/d\mu. 
$$ 
We call the measure $\mu$ \emph{physical for $f$} if there exists a basin $B_{\mu}$ of $\mu$ for $f$ and a measurable set $B$ so that $B\subset B_{\mu}$ and $\lambda(B)>0$.\\
An invariant measure $\hat{\nu}$ for the natural extension $\hat f:\hat X\to \hat X$ is called \emph{inverse limit physical measure} if $\hat{\nu}$ has a basin $\hat B_{\hat{\nu}}$ so that  $\lambda(\pi_0(\hat B_{\hat{\nu}}))>0$.
\end{defn}

\begin{thm}[Theorems 11 and 12 from \cite{KRS}]
If $\mu$ is a physical measure for $f:X\to X$ where $X$ is an Euclidean space, then the induced measure $\hat \mu$ on $\hat X$ is an inverse limit physical measure for the natural extension $\hat{f}$. In particular, there is a basin $\hat B_{\hat \mu}:=\pi_0^{-1}(B)$ of $\hat\mu$ for $\hat f$ with $\lambda(B)>0$.
\end{thm}

Let $M(I)$ be the space of
Borel probability measures on $I$ equipped with the {\it Prokhorov metric} $D$ defined by
$$
D(\mu, \nu):=\inf\left\{\eps>0\colon
\begin{array}{l l}
& \mu(A) \le \nu(B(A,\eps))+\eps \text{ and }
 \nu(A) \le \mu(B(A,\eps))+\eps \\
& \text{ for any Borel subset } A \subset I
\end{array}
\right\}
$$
for $\mu, \nu \in M(I)$. The following (asymmetric) formula
$$D(\mu, \nu)=\inf\{\eps>0\colon
\mu(A)\leq \nu(B(A,\eps))+\eps \text{ for all Borel subsets } A\subset I\}$$
is equivalent to original definition, which means we need to check only one of the inequalities.
It is also well known, that the topology induced by $D$ coincides with the {\it weak$^*$-topology} for measures, in particular $M(I)$ equipped with the metric $D$ is a compact metric space (for more details on Prokhorov metric and weak*-topology the reader is referred to \cite{Huber}).
 
 \subsection{Main construction}
In what follows, we will adjust Oxtoby-Ulam technique of full Lebesgue measure transformation \cite{OxtobyUlam} to the context of homotopies in parametrized BBMs and combine this with Brown's approximation theorem on a complete space to get a parametrized family of homeomorphisms with attractors that attract background physical Oxtoby-Ulam measure. Additionally, these attractors are varying continuously in the Hausdorff metric. 
Recall that a Borel probability measure on a manifold $M$ is called {\em Oxtoby-Ulam (OU)} or {\em good} if it is non-atomic, positive on open sets, and assigns zero measure to the boundary of manifold $M$ (if it exists) \cite{APBook,OxtobyUlam}.
In our case, we will first construct a measure $\hat{\lambda}_f$ on the inverse limit $\underleftarrow{\lim}(D,f)$ using Lebesgue measure $\lambda$ on $D$, where the map $f$ is a near-homeomorphism of $D$
and identity on $\partial D$. Then we will find a homeomorphism $\Theta_f\colon \underleftarrow{\lim}(D,f)\to D$ and define a push-forward measure $\lambda_f=(\Theta_f)_* \hat{\lambda}_f$.
By this construction it is clear that $\lambda_f$ is an OU-measure.

To provide a parametrized version of Brown's theorem and in particular to construct a continuously varying family of homeomorphisms $\Theta_f$, we need the following definitions.

\begin{defn}\label{def:LebesgueSequence}
Let $\underleftarrow{\lim}(X_i,f_i)$ be an inverse limit where $\{X_i\}_{i\geq 0}$ are continua and $\{f_i: X_{i+1}\to X_i\}_{i\geq 0}$ a collection of continuous maps. A sequence $(a_i)_{i\geq 0}$ of positive real numbers is a {\em Lebesgue sequence} for $\underleftarrow{\lim}(X_i,f_i)$ if there is a sequence $(b_i)_{i\geq 0}$ of positive real numbers such that:
\begin{enumerate}
    \item $\sum^{\infty}_{i=0} b_i<\infty$,
    \item for any $x_i,y_i\in X_j$ and any $i<j$, if $|x_i-y_i|<a_j$, then $|f_{i+1}\circ\ldots \circ f_{j}(x_i)-f_{i+1}\circ\ldots \circ f_{j}(y_i)|<b_j$.
\end{enumerate}
A sequence $(c_i)_{i\geq 0}$ of positive real numbers is a {\em measure sequence} for $\underleftarrow{\lim}(X_i,f_i)$ if:
\begin{enumerate}
    \item $\sum^{\infty}_{i=n+1} c_i<c_n/2$ for any $n\geq 0$,
    \item for any two points $\hat x\neq\hat y\in \underleftarrow{\lim}(X_i,f_i)$ there exists a non-negative integer $N$ so that $|x_{N+1}-y_{N+1}|>c_N$.
\end{enumerate}
\end{defn}

Now we are ready to prove the main theorem of this section.

\begin{proof}[Proof of Theorem \ref{lem:BBM}]
Let $C_\lambda(I)$ be the family of Lebesgue measure-preserving maps and let $Q=\cap^{\infty}_{n=0} Q_n\subset C_\lambda(I)$ be the intersection of open dense sets $Q_n$ satisfying Theorem~\ref{thm:UniLimPresLeb} (crookedness), Theorem~9 from \cite{BT} (leo property), Theorem~15 from \cite{BT} (weak mixing with respect to $\lambda$) and Theorem~3 from \cite{BCOT} (shadowing property).
Take a countable collection of maps $\{f_i\}_{i=0}^\infty \subset Q$ that are dense in $Q$. By assumptions we know that each of these maps is leo, Lebesgue measure is ergodic measure for each $f_i$ (it is even weakly mixing), has the shadowing property and by Theorem~\ref{thm:UniLimPresLeb} for every $\delta>0$ there exists $N\in \mathbb{N}$ so that $f^N_i$ is $\delta$-crooked.
Let $S\subset I$ be a set of full Lebesgue measure such that any $x\in S$ is generic point of all $f_i$.

We also fix a sequence $\{b_n\}_{n\in\N_0}\subset \mathbb{R}$ such that $\sum_{n=0}^\infty b_n<+\infty$, see Definition~\ref{def:LebesgueSequence}. Let $D\subset \mathbb{R}^2$ be a closed disc and $R\colon D\to D$ a homeomorphism that will be specified in two paragraphs.
We define a sequence $R^n_i$ of homeomorphisms of $D$ such that $\lim_{n\to\infty}R^n_{i_n}= R$ (for each sequence $i_n$) and $\rho(R^{n}_i,R)<\xi^n_i/4<b_n/8$ where the sequence $\xi^n_i$ will be specified later (we will need its faster convergence).
Note that $\rho(R^n_i,R^{n+1}_j)<\xi^n_i/4+\xi^{n+1}_j/4<\xi^n_i/2$ provided that $\xi^n_i>\xi^{n+1}_j$. All the lower indices will be specified later.

Note that for each $n$ there exist indices $\{j_i^n\}_{{i}\in\N}$ and $\delta_{i}^n<2^{-n}$ such that if we denote $A^n_i:=B(f_{j_i^n},\delta_i^n)$
then $A^n:=\cup_i A^n_i$ is dense (and open by the definition) in $Q$, $\{f_i\}_{i=0}^\infty\subset A^n$ and $A^n_i\cap A^n_j=\emptyset$ for $i\neq j$.
Simply, for any $i$ the set $\set{\rho(f_i,f_j) : j\neq i}$ is countable, and so we can choose $\delta_i^n$ outside this set, making construction of consecutive $A^n_i$ possible by induction (none of $f_j$'s is in the boundary of $A^n_i$, and we can indeed avoid the boundary since it is of the negligible size). We can also make each $\delta^n_i$ arbitrarily small, in particular for $m>n$ and any $i,j$
we may require that if $A^m_i\cap A^n_j\neq \emptyset$, then $A^m_i\subset A^n_j$.
We can also require that each $A^n_i\subset Q_n$, since each $Q_n$ is open and contains all functions $f_i$.

In our construction we will implement additional requirements on values of $\delta^n_i$, because we will need them to be sufficiently small as will be specified later.

For a closed disk $D\subset \mathbb{R}^2$, let $I\subset \Int (D)$ be the unit interval on which the BBM construction will take place, let $I\subset \Int(D_1)\subset D_1 \subset \Int (D_2)\subset D_2\subset \Int(D)$ where $D_1$ and $D_2$ are two closed discs and let $R\colon D\to D$
be a near-homeomorphism, such that $R(D_2)=I$, $R|_{D\setminus D_2}$ is one-to-one and $R$ is identity on the boundary of $D$.
We also require that the smash $R$ is done along radial lines. It is not hard to provide an analytic formula defining $R$.
These maps and discs are fixed throughout the whole construction.

Now let us briefly recall how BBM construction is performed in general, for more detail see Section~\ref{sec:arc}.
Given a map $f\in C_\lambda(I)$ we construct an unwrapping $\bar f\colon D\to D$ in the following way:
\begin{enumerate}
	\item $\bar f(I)\subset \Int D_1$ and as usually in BBMs $\bar f|_I$ is a rotated graph of $f$,
	\item $\bar f$ is identity on $D\setminus D_1$,
	\item $R\circ \bar{f}|_I=f$ and every point in $\Int D$ is attracted to $I$ under iteration of $R\circ \bar{f}$
	where $I$ is identified with $I$ in a standard way.
\end{enumerate}
We also denote $\tilde f=\bar f|_I$. One of the main features of the construction will be to ensure that unwrappings within the family that we construct vary continuously with $f$.

One important property to notice here is that any set $U\subset D_2\setminus D_1$ of positive Lebesgue measure in $D$ is transformed
onto set $R\circ \bar{f}(U)\subset I$ of positive Lebesgue measure on $I$. It is a consequence of Fubini's theorem, since the smash $R$ is performed along radial lines, and so the base of integration needs to have positive Lebesgue measure. 
 
Let $\set{C_i}_{i=1}^\infty\subset \Int D_1$ be a collection  of Cantor sets such that $C_i\cap I = \emptyset$, $C_i\cap C_j=\emptyset$ and $\lambda(D_1)={\sum_{i=1}^\infty}\lambda(C_i)$.
In other words, these Cantor sets fill densely $\Int D_1\setminus I$ and carry full Lebesgue measure. Such family of Cantor sets can be chosen using standard arguments. We may require that $\cup_i R(C_i)\subset S$, because the union of radial lines over $S$ has the full Lebesgue measure.

Now it is a good moment to set the first restriction on $\delta_i^0$. When $f_{j_i^0}$ is fixed, we also have $\bar f_{j_i^0}$ and so the images $\bar f_{j_i^0}(C_k)$ are explicitly determined as well.
Therefore, we may require that 
$$\dist(\bar f_{j_i^0}(I), \bar f_{j_i^0}(C_1))>4\delta_i^0, \dist(\bar f_{j_i^0}(I), \partial D_1)>4\delta_i^0 \text{ and } \dist( \partial D_1, \bar f_{j_i^0}(C_1))>4\delta_i^0.$$ 
Set 
$$h^0_{i}:=\bar f_{j_i^0}.$$

Let $a^0_{i}>0$ be such that if $d(x,y)<a^0_{i}$ for $x,y\in D$ then $d(h^0_{i}(x),h^0_{i}(y))<b_0/2$. We also require that $16\delta_i^0<b_0$. This implies that if map $H: D\to D$ satisfies $\rho(H, h^0_{i})<4\delta_i^0$ and $d(x,y)<a^0_{i}$
then 
\begin{eqnarray}
d(H(x),H(y)) &\leq&  d(h^0_{i}(x),h^0_{i}(y))+d(H(x),h^0_{i}(x))+d(h^0_{i}(y),H(y))<\nonumber\\
&<& b_0/2+8\delta_i^0<b_0.\label{eq:b0}
\end{eqnarray}
Finally, we require that
$$
d(x,y)<a^0_{i}\quad \Rightarrow \quad d(R_{0}^0\circ h^{0}_{i}(x),R_{0}^0\circ h^{0}_{i}(y))<b_0.
$$

Now let us explain how the first step of the construction will be made, the reader is also referred to Figure~\ref{fig:transform}.
Let us take any $k,l$ such that $A^1_k\subset A^0_l$. Then $f_{j_k^1}\in B(f_{j_l^0},\delta_l^0)$.
Construct a homeomorphism $G\colon D\to D$ such that $G(x)=x$ for all $x\not\in B(h^0_{l}(I),2\delta_l^0)$
and for $x\in I$ we have 
$$G(h^0_{l}(x))=G(\bar f_{j_l^0}(x))=\bar f_{j_k^1}(x).$$ 
Additionally we can require from the construction that $\rho(G,\textrm{id})<2\delta^0_{l}$, because we move the graph along horizontal lines. Similarly, we construct a map $H\colon D\to D$
such that $H(x)=x$ for $x\notin B(\bar f_{j_l^0}(C_1),\delta_l^0)$, 
$$H(\bar f_{j_l^0}(C_1))\subset {\bigcup^{\infty}_{i=1}} C_{i} \text{ and } d(H(x),x)<\delta_l^0$$
for all other $x\in B(\bar f_{j_l^0}(C_1), \delta_l^0)$.

\begin{figure}
\begin{tikzpicture}[scale=0.65]
\draw[solid](0,0) circle (3);
\draw[ultra thick](-1,0)--(1,0);
\node at (2.4,2.4) {\tiny $D$};
\node at (0.5,-0.3) {\tiny  $I$};
\draw[solid](0,0) circle (2.5);
\node at (1.95,1.95) {\tiny $D_2$};
\draw[solid](0,0) circle (2);
\node at (1.6,1.6) {\tiny $D_1$};
\draw[fill=blue!5,solid](0,0) circle (2);
\draw[ultra thick](-1,0)--(1,0);
\node at (0.3,-0.4) {\tiny $I$};
\draw[blue,thick] (-1,0.2)--(1/3,0.2)--(-1/3,-0.2)--(1,-0.2);
\node at (-0.35,0.65) {\color{blue} \small $\bar f_{j_l^0}$};
\draw[->,thick] (-2.9,0)--(-2.5,0);
\draw[->,thick] (2.9,0)--(2.5,0);
\draw[->,thick] (0,-2.9)--(0,-2.5);
\draw[->,thick] (0,2.9)--(0,2.5);
\draw (-2.5,0)--(0,0);
\draw (2.5,0)--(0,0);
\draw(0,-2.5)--(0,0);
\draw (0,2.5)--(0,0);
\draw (-1.3,2.13)--(-1/3,0);
\draw (-1.3,-2.13)--(-1/3,0);
\draw (1.3,2.13)--(1/3,0);
\draw (1.3,-2.13)--(1/3,0);
\draw (-2.15,1.3)--(-2/3,0);
\draw (-2.15,1.3)--(-2/3,0);
\draw (2.15,-1.3)--(2/3,0);
\draw (-2.15,-1.3)--(-2/3,0);
\draw (2.15,1.3)--(2/3,0);
\draw (2.5,0)--(2,0);
\draw(0,-2.5)--(0,-2);
\draw (0,2.5)--(0,2);
\node at (-0.5,1.5) {\tiny $R$};
\node at (0.4,0.9) {\color{red}\tiny $C_1$};
\draw[->] (3.1,0)--(3.6,0);
\node at (3.4,0.5) {\small $G$};
\draw [fill,red] (-0.05,1.4) rectangle (0.05,1.3);
\draw [fill,red] (-0.05,1.2) rectangle (0.05,1.1);
\draw [fill,red] (-0.05,0.8) rectangle (0.05,0.7);
\draw [fill,red] (-0.05,0.6) rectangle (0.05,0.5);
\end{tikzpicture}
\hspace{-0.3cm}
\begin{tikzpicture}[scale=0.65]
\draw[solid](0,0) circle (3);
\draw[ultra thick](-1,0)--(1,0);
\node at (2.4,2.4) {\tiny $D$};
\node at (0.5,-0.3) {\tiny  $I$};
\draw[solid](0,0) circle (2.5);
\node at (1.95,1.95) {\tiny $D_2$};
\draw[solid](0,0) circle (2);
\node at (1.6,1.6) {\tiny $D_1$};
\draw[fill=blue!5,solid](0,0) circle (2);
\draw[ultra thick](-1,0)--(1,0);
\node at (0.3,-0.4) {\tiny $I$};
\draw[teal,thick] (1,-0.2)--(0,-0.2)--(1/2,-0.1)--(-1/2,-0.1)--(1/2,0)--(-1/2,0.1)--(0,0.2)--(-1,0.2);
\node at (-0.3,0.65) {\color{teal} \small $\bar f_{j_k^1}$};
\node at (0.7,1.1) {\color{red}\tiny $\bar f_{j_l^0}(C_1)$};
\draw[->,thick] (-2.9,0)--(-2.5,0);
\draw[->,thick] (2.9,0)--(2.5,0);
\draw[->,thick] (0,-2.9)--(0,-2.5);
\draw[->,thick] (0,2.9)--(0,2.5);
\draw (-2.5,0)--(0,0);
\draw (2.5,0)--(0,0);
\draw(0,-2.5)--(0,0);
\draw (0,2.5)--(0,0);
\draw (-1.3,2.13)--(-1/3,0);
\draw (-1.3,-2.13)--(-1/3,0);
\draw (1.3,2.13)--(1/3,0);
\draw (1.3,-2.13)--(1/3,0);
\draw (-2.15,1.3)--(-2/3,0);
\draw (-2.15,1.3)--(-2/3,0);
\draw (2.15,-1.3)--(2/3,0);
\draw (-2.15,-1.3)--(-2/3,0);
\draw (2.15,1.3)--(2/3,0);
\draw (2.5,0)--(2,0);
\draw(0,-2.5)--(0,-2);
\draw (0,2.5)--(0,2);
\draw[->] (3.1,0)--(3.6,0);
\node at (3.4,0.5) {\small $H$};
\draw [fill,red] (-0.5,1.4) rectangle (-0.6,1.3);
\draw [fill,red] (-0.2,1.2) rectangle (-0.3,1.1);
\draw [fill,red] (0.2,0.8) rectangle (0.3,0.7);
\draw [fill,red] (0.5,0.6) rectangle (0.6,0.5);
\end{tikzpicture}
\hspace{-0.3cm}
\begin{tikzpicture}[scale=0.65]
\draw[solid](0,0) circle (3);
\draw[ultra thick](-1,0)--(1,0);
\node at (2.4,2.4) {\tiny $D$};
\node at (0.5,-0.3) {\tiny  $I$};
\draw[solid](0,0) circle (2.5);
\node at (1.95,1.95) {\tiny $D_2$};
\draw[solid](0,0) circle (2);
\node at (1.6,1.6) {\tiny $D_1$};
\draw[fill=blue!5,solid](0,0) circle (2);
\draw[ultra thick](-1,0)--(1,0);
\node at (0.3,-0.4) {\tiny $I$};
\draw[teal,thick] (1,-0.2)--(0,-0.2)--(1/2,-0.1)--(-1/2,-0.1)--(1/2,0)--(-1/2,0.1)--(0,0.2)--(-1,0.2);
\node at (-0.3,0.65) {\color{teal} \small $\bar f_{j_k^1}$};
\node at (0.5,1.1) {\color{red}\tiny $H(\bar f_{j_l^0}(C_1))$};
\draw[->,thick] (-2.9,0)--(-2.5,0);
\draw[->,thick] (2.9,0)--(2.5,0);
\draw[->,thick] (0,-2.9)--(0,-2.5);
\draw[->,thick] (0,2.9)--(0,2.5);
\draw (-2.5,0)--(0,0);
\draw (2.5,0)--(0,0);
\draw(0,-2.5)--(0,0);
\draw (0,2.5)--(0,0);
\draw (-1.3,2.13)--(-1/3,0);
\draw (-1.3,-2.13)--(-1/3,0);
\draw (1.3,2.13)--(1/3,0);
\draw (1.3,-2.13)--(1/3,0);
\draw (-2.15,1.3)--(-2/3,0);
\draw (-2.15,1.3)--(-2/3,0);
\draw (2.15,-1.3)--(2/3,0);
\draw (-2.15,-1.3)--(-2/3,0);
\draw (2.15,1.3)--(2/3,0);
\draw (2.5,0)--(2,0);
\draw(0,-2.5)--(0,-2);
\draw (0,2.5)--(0,2);
\draw [fill,red] (-0.9,1.4) rectangle (-1,1.3);
\draw [fill,red] (-0.8,1.2) rectangle (-0.9,1.1);
\draw [fill,red] (0.63,0.8) rectangle (0.73,0.7);
\draw [fill,red] (0.53,0.6) rectangle (0.63,0.5);
\end{tikzpicture}
\caption{The figure shows how maps $G$ and $H$ from the proof of Theorem~\ref{lem:BBM} transform $D$. Namely, the map $G$ switches to a different unwrapping which moves the Cantor set $C_1$ presumably away from the radial lines drawn in the picture. However, the map $H$ places this Cantor set $C_1$ to the appropriate position (possibly to different radial lines).}\label{fig:transform}
\end{figure}
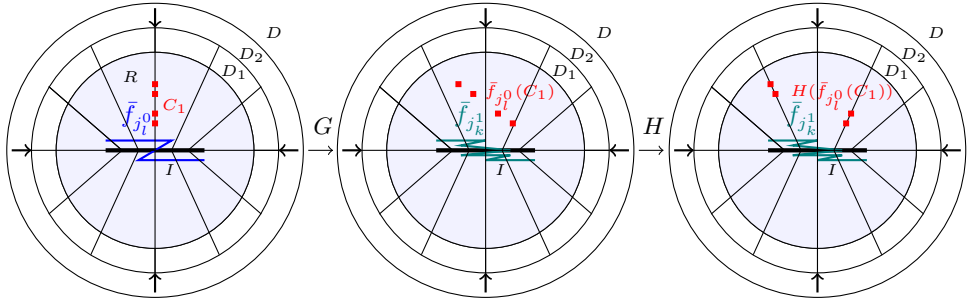

Simply, we first find Cantor sets which are in $\bigcup{_{i=1}^\infty} C_{i}$ and sufficiently well approximate small portions covering Cantor set $H(\bar f_{j_l^0}(C_1))$,
then define maps between these small portions and selected Cantor sets, and then extend the map to a homeomorphism on sufficiently small neighborhoods
where these translations of small portions take place. This is possible, because our Cantor sets are in the plane, so are \emph{tamely embedded}\footnote{A Cantor set $C$ in a manifold $M$ is \emph{tamely embedded} if there exist arbitrarily small neighbourhoods of $C$ that are finite unions of pairwise disjoint closed cubes from $M$.} (for such a construction see e.g. Appendix A in \cite{Cro}).
Now we can define the homeomorphism:
$$
h_{k}^1:=H\circ G\circ h^0_{l}.
$$
By our construction $\rho(h_{k}^1, h_{l}^0)<4\delta_l^0$,
$h_{k}^1|_I=\tilde f_{j_k^1}$ and $h_{k}^1(C_1)\subset \bigcup{_{i=1}^{\infty}} C_{{i}}$.
Similarly as above, we decrease $\delta^1_k$ so that 
\begin{equation}\label{eq:dist1}
\dist(h^1_{k}(I), h^1_{k}(C_1))>4\delta_k^1,\hspace{0.1cm} \dist(h^1_{k}(I),\partial D_1)>4\delta_k^1,\hspace{0.1cm} \dist( \partial D_1, h^1_{k}(C_2))>4\delta_k^1
\end{equation}
and additionally
\begin{equation}\label{eq:dist2}
\dist(h^1_{k}(I), h^1_{k}(C_2))>4\delta_k^1,\hspace{0.2cm} \dist( h^1_{k}(C_1), h^1_{k}(C_2))>4\delta_k^1.
\end{equation}
In the next step, the modifications will take place around the graph of $h_{k}^1(I)$ and Cantor set $h^1_{k}(C_2)$ so we need the neighborhoods of sets disjoint (\eqref{eq:dist1} and \eqref{eq:dist2}) and we do not want to change the definition of $h^1_{k}$ over $C_1$. 

This construction can be extended by induction in the following way. If $\delta_l^n>0$ is fixed sufficiently small with respect to the continuity of $R^n_l,R^{n+1}_k$ then
if $A{^n_l}\supset A{^{n+1}_k}$
then the maps $h_{l}^n$, $h_{k}^{n+1}$ satisfy:
\begin{enumerate}
\item $\rho(R_{l}^{n}\circ h_{l}^{n},R_{k}^{n+1}\circ h_{k}^{n+1})<2\rho(R_{l}^{n},R_{k}^{n+1})<\xi_{l}^{n}/2$,
\item $h_{k}^{n+1}|_I=\tilde f_{j_k^{n+1}}$,
\item $h_{l}^{n}|_{C_r}=h_{k}^{n+1}|_{C_r}$ for $r=1,\ldots,n$,
\item $h_{k}^{n+1}(C_r)\subset \cup{_{i=1}^{\infty}} C_{i}$ for $r=1,\ldots, n+1$,
\item $\delta_k^{n+1}<\delta_{l}^n/2$ and $\xi_k^{n+1}<\xi_{l}^n/2$.
\end{enumerate}

Now we will perform an additional adjustment of the constants $\delta^n_i$ and $\xi^n_i$ so that we are able to repeat arguments from \cite{Bro}. That is, we are going to ensure that there exist a so-called Lebesgue sequence and a measure sequence from Definition~\ref{def:LebesgueSequence} for the maps we construct.

Assume that the set $A^{n}_{j_n}$ is already constructed for some index $j_n\in \N$ and $A_{j_1}^1\supset A_{j_2}^2\supset \ldots \supset A_{j_n}^n$. Let $h_{j_i}^{i}$ be a homeomorphism corresponding to $A^i_{j_i}$.
There exists a positive real number  $a^n_{j_n}$ such that if $d(x,y)<a_{j_n}^n$ then
$$
 \forall i<n \quad \; d(R^{i+1}_{j_{i+1}}\circ h^{i+1}_{j_{i+1}}\circ \ldots \circ R^n_{j_{n}}\circ h_{j_n}^n(x),R^{i+1}_{j_{i+1}}\circ h^{i+1}_{j_{i+1}}\circ \ldots \circ R^n_{j_n}\circ h_{j_n}^n(y))<b_n.
$$
We require that $\delta_{j_{n+1}}^{n+1}$ for $A^{n+1}_{j_{n+1}}\subset A_{j_n}^n$ is adjusted with the correspondence to the 
condition $\xi^{n+1}_{j_{n+1}}<a_{j_n}^n/8$
. This will ensure that if uniform limit $R^n_{j_n}\circ h^n_{j_n}\to F$ exists, then
\begin{equation}
\rho(F,R^{n+1}_{j_{n+1}}\circ h^{n+1}_{j_{n+1}})\leq 4\xi^{n+1}_{j_{n+1}}\sum_{i=1}^\infty 2^{-i} \leq 8\xi^{n+1}_{j_{n+1}} <a_{j_n}^n
\label{eq:Fandhj}
\end{equation}
while $(a_{j_n}^n)_{n=1}^\infty$ is a Lebesgue sequence for $\{R^n_{j_n}\circ h_{j_n}^n\}_{n=1}^\infty$ and $(b_n)_{n=1}^\infty$.

Assume that a map  $T\in C(D,D)$ is given a priori and $F$ was obtained as its perturbation. Fix any $i,n>0$ and let $\gamma>0$ be such that $\rho(F,T)<\gamma$. Then for any $x\in D$ we have
$$
d(F^i(x),T^i(x))\leq \sum_{j=0}^{i-1} d(T^j(F^{i-j}(x)), T^{j+1}(F^{i-j-1}(x))),
$$
therefore, since $T$ is fixed, we have $\rho(F^i,T^i)<1/n$ for each $i=0,1,\ldots,n$, provided that $\gamma$ was sufficiently small ($\gamma$ depends on the continuity of $T,T^2,\ldots, T^n$).
Therefore, taking $\delta_{j_n}^n$, $\xi_{j_n}^n$ sufficiently small, we can require that if uniform limit $R^n_{j_n}\circ h^n_{j_n}\to F$ exists, then
$$
\rho(F^i,(R^n_{j_n}\circ h^n_{j_n})^i)<1/4n
$$
for $i=0,1,\ldots,n$. For each $j_n\in \N$ we pick a real number $c_{j_n}^n>0$ in such a way that $c_{j_n}^n<\frac{1}{8} c_k^{n-1}$  
where $A_{k}^{n-1}\supset A_{j_n}^n$ and additionally, if $d(x,y)<c_{j_n}^{n}$ for some $x,y\in D$, then 
$$
d((R^n_{j_n}\circ h^n_{j_n})^i(x),(R^n_{j_n}\circ h^n_{j_n})^i(y))<1/4n.
$$
This choice has the following consequences. First of all, 
$$\sum^\infty_{i={n+1}} c^{i}_{j_i}\leq \sum^\infty_{i={n+1}}8^{n-i} c_{j_n}^{n}<c_{j_n}^{n}/2.$$
Additionally, if we pick any distinct $\hat x,\hat y\in \underleftarrow{\lim}(D,F)$ then there is $M\in \N_0$ such that $d(x_{M},y_{M})>\gamma$ for some $\gamma>0$.
Take $n>M$ such that $1/n<\gamma$. Then 
\begin{gather*}
d(x_M,y_{M})\leq d(F^{n-M}(x_n),F^{n-M}(y_n))\leq\\
2\rho(F^{n-M}, (R^n_{j_n}\circ h^n_{j_n})^{n-M})+d((R^n_{j_n}\circ h^n_{j_n})^{n-M}(x_{n+1}),(R^n_{j_n}\circ h^n_{j_n})^{n-M}(y_{n+1}))\leq\\
1/2n+d((R^n_{j_n}\circ h^n_{j_n})^{n-M}(x_{n+1}),(R^n_{j_n}\circ h^n_{j_n})^{n-M}(y_{n+1})).
\end{gather*}
If $d((R^n_{j_n}\circ h^n_{j_n})^{n-M}(x_{n+1}),(R^n_{j_n}\circ h^n_{j_n})^{n-M}(y_{n+1}))<1/4n$ then we have a contradiction with the choice of $\gamma$,
therefore $d(x_{n+1},y_{n+1})>c_{j_n}^n$, meaning that $\{c_{j_n}^n\}_{n=1}^\infty$ is a measure sequence for $F$.
We additionally require that for each $n\in \N_0$ we have
\begin{equation}
8\delta^{n+1}_{j_{n+1}}<\min \{c_{j_n}^n, \min_{k<n} \{L(c^n_{j_n}, R^k_{j_k}\circ h_{j_k}^k\circ \ldots \circ R^n_{j_n}\circ h_{j_n}^n)\}\}
\label{eq:deltanL}
\end{equation}
where 
$$
L(\eps,G):=\sup \{\delta>0 : d(x,y)<\delta \quad \implies \quad d(G(x),G(y))<\eps\}.
$$
The above conditions are easily satisfied by induction.

Now we will turn our attention to the implications of the construction. Assume that the above inductive construction has been performed and fix any $g\in A=\cap_{n=1}^{\infty} A^n$.
Then there are indices  $i=i_n$ such that $g\in A^n_{i}=B(f_{j_{n}},\delta_{j_{n}}^n)$ where $j_n:=j^{n}_{i_n}$.
Consider the associated sequence of homeomorphisms $h^n_{j_n}\colon D\to D$.
For any $n<m$ we have
$$
\rho(h_{j_n}^{n},h_{j_m}^{m})<4\sum_{i=n}^{m-1}2^{-i}\leq 2^{-n+3},
$$
and therefore the maps $h^n_{j_n}$ form a Cauchy sequence in $C(D,D)$. Thus there exists a map $F_g$
obtained as the uniform limit of the maps $h^n_{j_n}$. But then $F_g|_I=\tilde g$ as $\tilde g$ is a uniform limit of maps $\tilde f_{j_n}=h^n_{j_n}|_I$.
Furthermore $h^n_{j_n}|_{C_r}=h^r_{j_r}|_{C_r}$ for all $n\geq r$
and therefore $F_g(C_r)\subset \cup{_{i=1}^{\infty}}C_i$.

Let us define a map $\mathcal{F}\colon D\times A\to D\times A$ by 
$$\mathcal{F}(x,g)=(R(F_g(x)),g).$$
Note that for every $f,g\in A$ and $\eps>0$ there is $\delta>0$ such that if $\rho(f,g)<\delta$ then there are $n$ and $j_n$
such that $2^{-n+4}<\eps$ and additionally
$\rho(F_f,h_{j_n}^{n})<4\delta_{j_n}^n<2^{-n+2}$ and $\rho(F_g,h_{j_n}^{n})<4\delta_{j_n}^n<2^{-n+2}$.
Namely, for sufficiently small $\delta$ we have $f,g\in B(g,\delta)\subset A^n_{j_n}$.
This shows that $\mathcal{F}$ is continuous.

Now we will deduce properties (c) and (d).
By the definition it holds that $\mathcal{F}(x,g)=(R(F_g(x)),g)=(\tilde g(x),g)$ for each $x\in I$.
If we fix any set of positive Lebesgue measure $U\subset D_2\setminus D_1$ then $R\circ F_g(U)=R(U)$ and $R(U)$ has positive one-dimensional Lebesgue measure on $I$.
But then by Fubini's theorem there is a set $S_g\subset U$ such that $\lambda(S_g)>0$ and $R(S_g)$ is contained in the set of generic points of $g$,
in particular any point $x\in S_g$ under iteration of $F_g$ {\em recovers the Lebesgue measure on $I$}, i.e. the measure $\frac{1}{n}\sum_{i=0}^n \delta_{(F^i_g(x))}$
converges in weak* topology to the Lebesgue measure on $I$.

But now consider the special case of map $F_i:=R\circ F_{f_i}$ and take any set $U\subset \Int D$ of positive Lebesgue measure. We can write $U=\cup_{j=0}^\infty U_j$
as a disjoint union of sets $U_j$ such that $j$ is the minimal index such that $F_i^j(U_j)\subset D_2$. Note that for any $j>0$ we have $F_i^j(U_j)=R^j(U_j)\subset D_2\setminus I$.
In particular, if $Y\subset F_i^j(U_j)$ is such that $\lambda(Y)=0$ then also $\lambda(R^{-j}(Y))=0$, where the latter formula makes sense, because $R^{-1}$
is well defined on  $D_2\setminus I$. But then $F_{f_i}(\bigcup_{j=0}^\infty R^j(\tilde{U}_j))\subset \cup{_{i=1}^{\infty}}C_i$ for some sets $\tilde U_j\subset U_j$ satisfying $\lambda(U_j)=\lambda(\tilde U_j)$ and therefore 
$$
F_i^{j+1}(\tilde{U}_j)\subset S
$$
for each $j$. But then there is a set $\tilde S_i$ of full Lebesgue measure in $D$ such that for each $x\in \tilde S_i$
there is $N\in \N_0$ such that $F_i^N(x)\in S$. This means that every point in $\tilde S_i$ is eventually transferred into a generic point of $f_i$,
which means that the orbit of $x$ under $F_i$ recovers the one-dimensional Lebesgue measure on $I$. This shows that the Lebesgue measure on $I$ is a physical measure for each $\mathcal{F}(\cdot,g)$
and it is unique physical measure for a dense set of functions $g\in C_{\lambda}(I)$ (this dense set corresponds with the maps $\{f_i\}^{\infty}_{i=0}$ from the start of the construction). In fact it is unique each time
when generic points of $g$ contain the set $S$ and may have (but not necessarily has) other physical measures in remaining cases.

Denote $\hat{D}:=\underleftarrow{\lim} (D\times A, \mathcal{F})$.
Now we are going to define a map $\Theta \colon \hat{D} \to D\times A$
by 
$$\Theta(\hat x,g):=(\lim_{n\to \infty} R^1_{i_1}\circ h^1_{i_1}\circ\ldots \circ R^n_{i_n}\circ h^n_{i_n}(x_n),g)$$

where $g\in \cap{_{n=1}^{\infty}} A^n_{i_n}$. We can write $g$ as the second coordinate in $\underleftarrow{\lim} (D\times A, \mathcal{F})$, since it is a constant sequence of $g$'s;  thus we can also write $\Theta_g:=\Theta(\cdot,g)\colon \hat D_g\to D$, where $\hat D_g:=\underleftarrow{\lim} (D\times \{g\}, \mathcal{F})$. Since we have satisfied its assumptions, by Theorem~1 from \cite{Bro} $\Theta$ is well defined. Furthermore, by Theorem~2 from \cite{Bro}, it holds that $\Theta(\cdot,g)$ is a homeomorphism for each $g\in A$, because it is a composition of a homeomorphism with projection onto the first coordinate in the inverse limit defined by homeomorphisms
$\underleftarrow{\lim} (D,R^n_{i_n}\circ h^n_{i_n})$.

Note that if $(\hat x,g)\in \hat D$
then
\begin{eqnarray*}
&&d(R^1_{i_1}\circ h^1_{i_1}\circ\ldots \circ R^n_{i_n}\circ h^n_{i_n}(x_n),R^1_{i_1}\circ h^1_{i_1}\circ\ldots \circ R^{n+1}_{i_{n+1}}\circ h^{n+1}_{i_{n+1}}(x_{n+1}))=\\
&&\hspace{-15pt} d(R^1_{i_1}\circ h^1_{i_1}\circ\ldots \circ R^n_{i_n}\circ h^n_{i_n}\circ R\circ F_g(x_{n+1}),R^1_{i_1}\circ h^1_{i_1}\circ\ldots \circ R^{n+1}_{i_{n+1}}\circ h^{n+1}_{i_{n+1}}(x_{n+1}))
\end{eqnarray*}

But by \eqref{eq:Fandhj} we have
$$
d(R^{n+1}_{i_{n+1}}\circ h^{n+1}_{i_{n+1}}(x_{n+1}),R\circ F_g(x_{n+1}))\le 8\xi^{n+1}_{i_{n+1}} 
$$
so by \eqref{eq:deltanL} we obtain
$$
d(R^1_{i_1}\circ h^1_{i_1}\circ\ldots \circ R^n_{i_n}\circ h^n_{i_n}(x_n),R^1_{i_1}\circ h^1_{i_1}\circ\ldots \circ R^{n+1}_{i_{n+1}}\circ h^{n+1}_{i_{n+1}}(x_{n+1}))<c_{j_n}^n
$$
Using telescoping sum, this gives for any $\eps>0$ and $n$ sufficiently large (here $\Theta(\hat x,g)_1$ denotes the natural projection to the first coordinate) 
$$
d(R^1_{i_1}\circ h^1_{i_1}\circ\ldots \circ R^n_{i_n}\circ h^n_{i_n}(x_n),\Theta(x,g)_1)<\sum_{l\geq n}c_{j_l}^l<2c_{j_n}^n<\eps
$$
Note that the previous estimate is true for any $g\in \cap^{\infty}_{i=1} A_{j_{i}}^l$
and every $\hat x$ such that $(\hat x,g)\in \hat D$. Thus as a consequence, for $\delta$ sufficiently small, all $f,g\in \cap^{\infty}_{i=1} A_{j_{i}}^l$ and $d(\hat x,\hat y)<\delta$ we have
$$
d(R^1_{i_1}\circ h^1_{i_1}\circ\ldots \circ R^n_{i_n}\circ h^n_{i_n}(x_n),R^1_{i_1}\circ h^1_{i_1}\circ\ldots \circ R^n_{i_n}\circ h^n_{i_n}(y_n))<\eps.
$$
As a result, under the above assumptions, we get 
$(\Theta(\hat x,g)_1),\Theta(\hat y,f)_1)<3\eps$
which proves that $\Theta$ is continuous.

For each $f$ we define a homeomorphism  $\Phi_f:=\Theta_f \circ \hat{\mathcal{F}}(\cdot,f) \circ \Theta_f^{-1}\colon D\to D$. Abusing the notation, for the following inverse limit spaces we will identify $f$ with the interval map $f|_{I}$. 
Denote $\Lambda_f:=\Phi_f(\hat I_f)$, where $\hat I_f:=\underleftarrow{\lim}(I,f)$ and note that by Corollary~\ref{cor:MincTransue} the attractor $\Lambda_f$ is the pseudo-arc for every $f\in A$.
We can write $\Phi:=(\Theta\times \id) \circ \hat{\mathcal{F}} \circ (\Theta\times \id)^{-1}$ and put $\Phi_f=\Phi(\cdot, f)$
showing that the family $\Phi_f$ is varying continuously. This also shows that the family $\{\Lambda_f\}_{f\in A}$ varies continuously in Hausdorff metric.

 By Theorem~\ref{thm:5.2}, let $\hat{\mu}_f$ be an invariant measure induced on the inverse limit $\hat I_f$ using Lebesgue measure $\lambda$ on $I$ and define a push-forward measure $\mu_f=(\Theta_f)_* \hat{\mu}_f$.
Formally, the measure $\hat{\mu}_f$ is defined on the space $\hat I_f\subset \underleftarrow{\lim}(D,f)$,
however we can also view it as a measure on the space $\hat{D}$.
Let us show that measures $\hat{\mu}_f$ vary continuously in the weak* topology in $\hat{D}$. By the definition $\hat \mu_f(\pi_n^{-1}(B))=\lambda(B\cap I)$ 
for any Borel set $B\subset D$ and every $n\in \mathbb{N}_0$ (this measure can be viewed in $\hat{D}$ on a ``slice'' defined by $f$; for more detail see \cite{BdCH1}).
Take any uniformly distributed finite set $P\subset I$ and for any interval map $g$
define a finite set $\hat{P}_g\subset \underleftarrow{\lim}(I,g)$ such that $\pi_n(\hat{P}_g)=P$. 
Denote $\nu:=\frac{1}{|P|}\sum_{q\in P}\delta_q$
and $\hat \nu_g:=\frac{1}{|\hat{P}_g|}\sum_{\hat q\in \hat{P}_g}\delta_{\hat q}$.
Fix any $\eps>0$ and let us assume that $2^{-n}<\eps/2$. 
There is $\gamma>0$
such that if $\rho(f,g)<\gamma$, $\hat x,\hat y\in \underleftarrow{\lim}(I,g)$
and $d(\pi_n(\hat x),\pi_n(\hat y))<2\gamma$ 
then $d(\hat x,\hat y)<\eps$.
We may also assume that for any two consecutive points $p,q\in P$ with respect to the ordering in $I$ we have $d(p,q)<\gamma/2$ and that $\gamma>0$ is sufficiently small so that if $\hat x\in \underleftarrow{\lim}(I,g)$, $\hat y\in \underleftarrow{\lim}(I,f)$
satisfy $\pi_n(\hat x)=\pi_n(\hat y)$, then $d(\hat x,\hat y)<\eps$.

Note that since $\rho(g,f)<\gamma$ we have
\begin{eqnarray*}
\hat \mu_g(\hat{B})\leq \lambda(\pi_n(\hat{B}))&\leq& \lambda(\cup_{q\in P,(q-\gamma,q+\gamma)\cap \pi_n(\hat{B})\neq \emptyset}(q-\gamma,q+\gamma))\\
&\leq& \nu(B(\pi_n(\hat{B}),2\gamma))=\hat{\nu}_g(\pi_n^{-1}(B(\pi_n(\hat{B}),2\gamma))
\end{eqnarray*}
If $\hat{q}\in \hat{P}$ satisfies $\hat{q}\in \pi_n^{-1}(B(\pi_n(\hat{B}),2\gamma)$
then there is $\hat{z}\in \hat{B}$ such that $d(q,\pi_n(\hat{z}))<2\gamma$
and therefore $d(\hat{z},\hat{q})<\eps$. This gives
$$
\hat \mu_g(\hat{B})\leq \hat{\nu}_g(B(\hat{B},\eps))
$$
and therefore $D(\hat \mu_g,\hat \nu_g)<\eps$ (recall that $D(\cdot,\cdot)$ denotes the Prokhorov metric on $M(I)$ defined in the end of Subsection~\ref{subsec:MTpreliminaries}).
Clearly, for every $\hat{q}\in \hat{P}_g$ there is $\hat{p}\in \hat{P}_f$ such that $\pi_n(\hat{q})=\pi_n(\hat{p})$
and therefore $D(\hat \nu_f,\hat \nu_g)<\eps$.
This gives that $D(\hat \mu_f,\hat \mu_g)<3\eps$
provided that $\rho(f,g)<\gamma$.
Indeed, the function $f\mapsto \hat{\mu}_f$ is continuous.

If $\alpha\in C(D, \mathbb{R})$, then by identifying $\Theta$ to the projection on the first coordinate $\alpha\circ \Theta\in  C(\hat D,\mathbb{R})$ and we have already proven that for any $f_i\to f$ from $A$,
$$
\int_{\hat D} \alpha \circ \Theta d\hat \mu_{f_i} \to \int_{\hat D} \alpha \circ \Theta d\hat \mu_{f}.
$$

We therefore have that 
$$
\int_{D} \alpha d \mu_{f_i} \to \int_{D} \alpha d \mu_{f}.
$$
This proves the continuity of the map $f\mapsto \mu_f$ where each $\mu_f$ is the push-forward measure on $D$ defined by $\mu_f:=\Theta_* \hat{\mu}_f$.
It is clear from the definition, that the support of $\mu_f$ is $\Lambda_f\subset D$.
\end{proof}

\section*{Acknowledgements}
J. \v C. was supported by the Austrian Science Fund (FWF) Schrödinger Fellowship stand-alone project J 4276-N35 and partially by the IDUB project no. 1484 ``Initiative for Excellence - Research University'' at AGH UST.
P. O. was supported by National Science Centre, Poland (NCN), grant no. 2019/35/B/ST1/02239. For the purpose of Open Access the authors have applied a CC BY public copyright licence to any Author Accepted Manuscript version arising from this submission.
 We are very grateful to J. P. Boro\'nski, H. Bruin, A. de Carvalho and J. C. Mayer for helpful comments on the early version of the paper.  We would like to thank the anonymous referee whose insightful suggestions substantially improved several parts of the paper.

\end{document}